\pdfoutput=1
\documentclass[reqno, 12pt]{amsart}
\def\ssign{\textsection\nobreak\hspace{1pt plus 0.3pt}}
\makeatletter
\@ifundefined{@secnumfont}{}{%
	\newif\if@sectsign \@sectsigntrue
	\def\@seccntformat#1{%
		\protect\textup{\protect\@secnumfont
			\ifnum\pdfstrcmp{#1}{section}=\z@
				\if@sectsign \protect\ssign\csname the#1\endcsname\protect\enspace
				\else \csname the#1\endcsname\protect\@secnumpunct\fi
			\else
				\csname the#1\endcsname\protect\@secnumpunct
			\fi
		}%
	}%
	\@ifundefined{appendix}{}{\g@addto@macro\appendix{\@sectsignfalse}}%
}
\makeatother
\makeatletter
\@namedef{section}{\@startsection{section}{1}%
	\z@{-.7\linespacing\@plus-\linespacing}{.5\linespacing}%
{\normalfont\scshape\centering}}
\g@addto@macro\ps@headings{\let\sectionmark\@gobble}
\let\sectionmark\@gobble
\makeatother

\makeatletter
\@ifclassloaded{amsbook}{%
	\usepackage{etoolbox}%
	\patchcmd{\chapter}{\@afterindenttrue}{\@afterindentfalse}{}{\ERROR}%
}{}
\makeatother

\usepackage[T1]{fontenc}
\usepackage{lmodern}
\usepackage[british]{babel}
\usepackage[babel]{microtype}

\usepackage{amsmath,amssymb,amsthm}
\usepackage{mathrsfs}
\usepackage{mathabx}\changenotsign
\usepackage{dsfont}
\usepackage{scalerel}
\usepackage{graphicx}
\usepackage{trimclip}

\usepackage{geometry}
\usepackage{xcolor}
\usepackage[backref]{hyperref}
\hypersetup{
	colorlinks,
	linkcolor={red!60!black},
	citecolor={green!60!black},
	urlcolor={blue!60!black}
}

\usepackage[open,openlevel=2,atend]{bookmark}

\usepackage[abbrev,msc-links,backrefs]{amsrefs}
\usepackage{doi}

\renewcommand{\PrintDOI}[1]{\doi{#1}}

\numberwithin{equation}{section}
\numberwithin{figure}{section}

\usepackage{calc}
\usepackage{enumitem}
\def\rmlabel{\upshape({\itshape \roman*\,})}

\def\alabel{\upshape(\makebox[\widthof{\itshape a}][c]{\itshape \alph*}\,)}

\theoremstyle{plain}
\newtheorem{thm}{Theorem}[section]
\newtheorem{fact}[thm]{Fact}
\newtheorem{prop}[thm]{Proposition}
\newtheorem{cor}[thm]{Corollary}
\newtheorem{lemma}[thm]{Lemma}

\theoremstyle{definition}
\newtheorem{dfn}[thm]{Definition}

\theoremstyle{remark}

\newtheorem*{nrem}{Remark}
\newtheorem{clm}[thm]{Claim}

\let\eps=\varepsilon
\let\theta=\vartheta
\let\rho=\varrho
\let\phi=\varphi

\let\polishlcross=\l
\protected\def\l{\ifmmode\ell\else\polishlcross\fi}
\ifdefined{\def\l{l}}\fi

\usepackage{textcase}

\makeatletter
\def\@defbb#1{\expandafter\def\csname#1#1\endcsname{{\mathds{#1}}}}
\def\@defcal#1{\expandafter\def\csname c#1\endcsname{{\mathcal{#1}}}}
\def\@defscr#1{\expandafter\def\csname cc#1\endcsname{{\mathscr{#1}}}}
\def\@deffrak#1{\expandafter\def\csname f#1\endcsname{{\mathfrak{#1}}}}
\@tfor\next:=BCDEFGHIJKLMNOPQRTUVWXYZ\do{\expandafter\@defbb\next}
\@tfor\next:=ABCDEFGHIJKLMNOPQRSTUVWXYZ\do{\expandafter\@defcal\next}
\@tfor\next:=ABCDEFGHIJKLMNOPQRSTUVWXYZ\do{\expandafter\@defscr\next}
\@tfor\next:=ABCDEFGHIJKLMNOPQRSTUVWXYZ\do{\expandafter\@deffrak\next}
\@tfor\next:=abcdefghjklmnopqrstuvwxyz\do{\expandafter\@deffrak\next}
\makeatother
\let\Aoverring=\AA
\protected\def\AA{\ifmmode{\mathds{A}}\else\Aoverring\fi}
\let\Eszett=\SS
\protected\def\SS{\ifmmode{\mathds{S}}\else\Eszett\fi}
\ifdefined
\fi

\makeatletter
\def\moverlay{\mathpalette\mov@rlay}
\def\mov@rlay#1#2{\leavevmode\vtop{   \baselineskip\z@skip
		\lineskiplimit-\maxdimen
\ialign{\hfil$\m@th#1##$\hfil\cr#2\crcr}}}
\newcommand{\charfusion}[3][\mathord]{
	#1{\ifx#1\mathop\vphantom{#2}\fi
		\mathpalette\mov@rlay{#2\cr#3}
	}
\ifx#1\mathop\expandafter\displaylimits\fi}
\makeatother

\newcommand{\dcup}{\charfusion[\mathbin]{\cup}{\cdot}}

\DeclareMathOperator{\girth}{girth}

\DeclareFontFamily{U}  {MnSymbolC}{}
\DeclareSymbolFont{MnSyC}         {U}  {MnSymbolC}{m}{n}
\DeclareFontShape{U}{MnSymbolC}{m}{n}{
	<-6>  MnSymbolC5
	<6-7>  MnSymbolC6
	<7-8>  MnSymbolC7
	<8-9>  MnSymbolC8
	<9-10> MnSymbolC9
	<10-12> MnSymbolC10
<12->   MnSymbolC12}{}
\DeclareMathSymbol{\powerset}{\mathord}{MnSyC}{180}

\def\tand{\ \text{and}\ }
\def\qand{\quad\text{and}\quad}
\def\qqand{\qquad\text{and}\qquad}

\let\setminus=\smallsetminus
\let\emptyset=\varnothing
\let\vn=\varnothing

\let\lra=\longrightarrow
\let\to=\lra

\makeatletter
\def\@shortto@sub_#1{_{\begingroup\let\to\rightarrow#1\endgroup}}
\newcommand{\@withshortto}[1]{#1\@ifnextchar_\@shortto@sub\relax}
\@for\@op:=lim,limsup,liminf,varlimsup,varliminf,varinjlim,varprojlim,projlim,injlim\do{%
	\@ifundefined{\@op}{}{%
		\expandafter\let\csname @save@\@op\expandafter\endcsname\csname\@op\endcsname
		\protected\expandafter\edef\csname\@op\endcsname{\noexpand\@withshortto\expandafter\noexpand\csname @save@\@op\endcsname}%
	}%
}
\makeatother

\makeatletter
\newcommand{\stmry@slash}[2]{%
	\hbox{\check@mathfonts
		\ifx#1\scriptstyle\fontsize{\sf@size}{\z@}\selectfont
		\else\ifx#1\scriptscriptstyle\fontsize{\ssf@size}{\z@}\selectfont
		\fi\fi
		\usefont{U}{stmry}{m}{n}\char#2}}
\newcommand{\arrownot}{\mathrel{\mathpalette\stmry@slash{"58}}}
\newcommand{\Arrownot}{\mathrel{\mathpalette\stmry@slash{"59}}}
\makeatother
\def\longarrownot{\mathrel{\mkern5.5mu\arrownot\mkern-5.5mu}}

\def\nlra{\longarrownot\longrightarrow}

\newlength\tarrowwd \newlength\tarrowmin \newlength\tarrowdp
\DeclareRobustCommand{\tarrow}[2][-1.6pt]{\settowidth\tarrowwd{$\scriptstyle#2$}%
	\addtolength\tarrowwd{10pt}\settowidth\tarrowmin{$\longrightarrow$}%
	\settodepth{\tarrowdp}{$\scriptstyle#2$}%
	\mathrel{\mathop{\ifdim\tarrowwd>\tarrowmin
		\hbox to \tarrowwd{\rightarrowfill}\else \longrightarrow \fi}%
		\limits^{\vphantom{f}\raisebox{\dimexpr#1+\tarrowdp\relax}[0pt][0pt]{$\scriptstyle#2$}}}}
\DeclareMathSymbol{\cstar}{\mathbin}{symbols}{"03}

\def\nilra{\xrightarrow{\textnormal{\,n.n.i.\,}}}

\newlength\tarrowht
\DeclareRobustCommand{\lrasub}[2][3.75pt]{\settowidth\tarrowwd{$\scriptstyle#2$}%
	\addtolength\tarrowwd{10pt}\settowidth\tarrowmin{$\longrightarrow$}%
	\settoheight\tarrowht{$\scriptstyle#2$}%
	\mathrel{\mathop{\ifdim\tarrowwd>\tarrowmin
		\hbox to \tarrowwd{\rightarrowfill}\else \longrightarrow \fi}%
		\limits_{\smash{\raisebox{\dimexpr#1+\fontdimen5\scriptfont1-\tarrowht\relax}%
			{$\scriptstyle#2$}}}}}

\makeatletter
\@ifpackageloaded{mathabx}{}{%
	\DeclareFontFamily{U}{matha}{\hyphenchar\font45}%
	\DeclareFontShape{U}{matha}{m}{n}{<-11> matha10 <11-> matha12}{}%
	\DeclareSymbolFont{matha}{U}{matha}{m}{n}%
}%
\DeclareMathSymbol{\hp@RU}{\mathord}{matha}{"E1}%
\DeclareMathSymbol{\hp@LU}{\mathord}{matha}{"E0}%
\newcommand{\hpsc}{0.6}\newcommand{\hpgapn}{3}\newcommand{\hpbot}{0.68}
\newcommand{\hpbarb}{0.55}\newcommand{\hpoh}{0.1}\newcommand{\hpminf}{0.5}\newcommand{\hpov}{0.04}\newcommand{\hpdx}{0}
\newcommand{\rhpinL}{0.1}\newcommand{\rhpinR}{0.0}
\newcommand{\lhpinL}{0.1}\newcommand{\lhpinR}{0.0}
\newsavebox\hpA\newsavebox\hpG\newsavebox\hpB\newsavebox\hpF\newsavebox\hpL\newsavebox\hpSl
\newdimen\hp@li\newdimen\hp@L\newdimen\hp@gap\newdimen\hp@dx\newdimen\hp@bs\newdimen\hp@oh\newdimen\hp@rw\newdimen\hp@ov\newdimen\hp@ri\newdimen\hp@xi\newif\ifhp@cramp
\def\hp@nil{}
\def\@hpfirst#1#2\hp@nil{\ifcat A\noexpand#1\sbox\hpF{$\m@th\SavedStyle#1$}\else\sbox\hpF{$\m@th\SavedStyle#1#2$}\fi}
\def\@hplast#1{\ifx#1\hp@nil\else\def\@hplastv{#1}\expandafter\@hplast\fi}
\newcommand{\hp@setA}[1]{%
	\sbox\hpA{$\m@th\SavedStyle\kern-\nulldelimiterspace\radical\z@{#1}$}%
	\dimen@=1.25\fontdimen8\textfont3
	\if D\m@switch \dimen@=\fontdimen8\textfont3 \advance\dimen@.25\fontdimen5\textfont2 \fi
	\if S\m@switch \dimen@=1.25\fontdimen8\scriptfont3 \fi
	\if s\m@switch \dimen@=1.25\fontdimen8\scriptscriptfont3 \fi
	\advance\dimen@-\ht\hpA \ht\hpA=-\dimen@}

\newcommand{\hp@common}[3]{%
	\ifhp@cramp \hp@setA{#3}\else \sbox\hpA{$\m@th\SavedStyle#3$}\fi
	\@hpfirst#3\hp@nil \@hplast#3\hp@nil \sbox\hpL{$\m@th\SavedStyle\@hplastv$}%
	\hp@li=#1\wd\hpF \hp@ri=#2\wd\hpL
	\hp@L=\wd\hpA \advance\hp@L-\hp@li \advance\hp@L-\hp@ri
	\hp@bs=\hpbarb\wd\hpG \hp@oh=\hpoh\ht\hpG
	\hp@xi=\fontdimen8\textfont3
	\if S\m@switch\hp@xi=\fontdimen8\scriptfont3\fi \if s\m@switch\hp@xi=\fontdimen8\scriptscriptfont3\fi
	\hp@dx=\hpdx\wd\hpA \hp@ov=\hpov\wd\hpG}
\newcommand{\hp@overgap}{\hp@gap=\hpgapn\hp@xi \@tempdima=\hpbot\ht\hpG \advance\hp@gap-\@tempdima}
\newcommand{\hp@overbox}[2]{\mathord{\vbox{\offinterlineskip\ialign{##\cr
	\hbox to\wd\hpA{\kern#1#2\hfil}\cr
	\noalign{\kern\hp@gap}\hbox{\usebox\hpA}\cr}}}}

\newcommand{\hpdrawR}{%
	\ifdim\hp@L<\wd\hpG
		\clipbox{\dimexpr\wd\hpG-\hp@L\relax{} 0pt 0pt \dimexpr-\hp@oh\relax{}}{\usebox\hpG}%
	\else
		\hbox to\dimexpr\hp@rw+\hp@ov\relax{\cleaders\hbox{\usebox\hpSl}\hfil}\kern-\hp@ov\usebox\hpB%
	\fi}
\newcommand{\hp@bR}[3]{\hp@cramptrue
	\sbox\hpG{\scalebox{\hpsc}{$\m@th\SavedStyle\hp@RU$}}%
	\hp@common{#1}{#2}{#3}%
	\ifdim\hp@L<\hpminf\wd\hpG \hp@L=\hpminf\wd\hpG \fi
	\hp@rw=\hp@L \advance\hp@rw-\wd\hpG \advance\hp@rw\hp@bs
	\sbox\hpB{\clipbox{\dimexpr\hp@bs\relax{} 0pt 0pt \dimexpr-\hp@oh\relax{}}{\usebox\hpG}}%
	\sbox\hpSl{\clipbox{\dimexpr\wd\hpG*1/5\relax{} 0pt \dimexpr\wd\hpG*39/50\relax{} 0pt}{\usebox\hpG}}%
	\hp@overgap \hp@overbox{\dimexpr\wd\hpA-\hp@ri-\hp@L+\hp@dx\relax}{\hpdrawR}}

\newcommand{\hpdrawL}{%
	\ifdim\hp@L<\wd\hpG
		\clipbox{0pt 0pt \dimexpr\wd\hpG-\hp@L\relax{} \dimexpr-\hp@oh\relax{}}{\usebox\hpG}%
	\else
		\usebox\hpB\kern-\hp@ov\hbox to\dimexpr\hp@rw+\hp@ov\relax{\cleaders\hbox{\usebox\hpSl}\hfil}%
	\fi}
\newcommand{\hp@bL}[3]{\hp@cramptrue
	\sbox\hpG{\scalebox{\hpsc}{$\m@th\SavedStyle\hp@LU$}}%
	\hp@common{#1}{#2}{#3}%
	\ifdim\hp@L<\hpminf\wd\hpG \hp@L=\hpminf\wd\hpG \fi
	\hp@rw=\hp@L \advance\hp@rw-\wd\hpG \advance\hp@rw\hp@bs
	\sbox\hpB{\clipbox{0pt 0pt \dimexpr\hp@bs\relax{} \dimexpr-\hp@oh\relax{}}{\usebox\hpG}}%
	\sbox\hpSl{\clipbox{\dimexpr\wd\hpG*39/50\relax{} 0pt \dimexpr\wd\hpG*1/5\relax{} 0pt}{\usebox\hpG}}%
	\hp@overgap \hp@overbox{\dimexpr\hp@li+\hp@dx\relax}{\hpdrawL}}
\NewDocumentCommand\rharp{O{\rhpinL}O{\rhpinR}m}{\ThisStyle{\hp@bR{#1}{#2}{#3}}}
\NewDocumentCommand\lharp{O{\lhpinL}O{\lhpinR}m}{\ThisStyle{\hp@bL{#1}{#2}{#3}}}
\makeatother

\def\longarrownotnni{\mathrel{\mkern10.5mu\arrownot\mkern-10.5mu}}
\def\ninlra{\longarrownotnni\xrightarrow{\textnormal{\,n.n.i.\,}}}

\usepackage{tikz}

\makeatletter
\@ifpackageloaded{tikz}{}{\RequirePackage{tikz}}
\makeatother
\usetikzlibrary{calc}

\newcommand{\tedgefull}[5][]{%
	\path[line join=round, #1]
	let \p1=($#4-#3$), \p2=($#5-#4$), \p3=($#3-#5$),
	    \n1={atan2(\y1,\x1)}, \n2={atan2(\y2,\x2)}, \n3={atan2(\y3,\x3)},
	    \n9={Mod(\n2-\n1,360)<180 ? 1 : -1}
	in
	($#3+(\n1-\n9*90:#2)$) -- ($#4+(\n1-\n9*90:#2)$)
	arc[start angle=\n1-\n9*90, delta angle={\n9*Mod(\n9*(\n2-\n1),360)}, radius=#2]
	-- ($#5+(\n2-\n9*90:#2)$)
	arc[start angle=\n2-\n9*90, delta angle={\n9*Mod(\n9*(\n3-\n2),360)}, radius=#2]
	-- ($#3+(\n3-\n9*90:#2)$)
	arc[start angle=\n3-\n9*90, delta angle={\n9*Mod(\n9*(\n1-\n3),360)}, radius=#2]
	-- cycle;
}

\newdimen\hvertexradius
\newcommand{\hedgeoffset}{2.25\hvertexradius}

\tikzset{
	hedge/.style={line width=1.25pt, fill opacity=0.2},
	tedge/.style={hedge, draw=red!70!black, fill=red!70!black},
	fedge/.style={hedge, draw=yellow!80!black, fill=yellow},
}

\newcommand{\tedge}[4][]{\tedgefull[tedge, #1]{\hedgeoffset}{#2}{#3}{#4}}

\newcommand{\fedgefull}[6][]{%
	\path[line join=round, #1]
	let \p1=($#4-#3$), \p2=($#5-#4$), \p3=($#6-#5$), \p4=($#3-#6$),
	    \n1={atan2(\y1,\x1)}, \n2={atan2(\y2,\x2)},
	    \n3={atan2(\y3,\x3)}, \n4={atan2(\y4,\x4)},
	    \n9={Mod(\n2-\n1,360)<180 ? 1 : -1}
	in
	($#3+(\n1-\n9*90:#2)$) -- ($#4+(\n1-\n9*90:#2)$)
	arc[start angle=\n1-\n9*90, delta angle={\n9*Mod(\n9*(\n2-\n1),360)}, radius=#2]
	-- ($#5+(\n2-\n9*90:#2)$)
	arc[start angle=\n2-\n9*90, delta angle={\n9*Mod(\n9*(\n3-\n2),360)}, radius=#2]
	-- ($#6+(\n3-\n9*90:#2)$)
	arc[start angle=\n3-\n9*90, delta angle={\n9*Mod(\n9*(\n4-\n3),360)}, radius=#2]
	-- ($#3+(\n4-\n9*90:#2)$)
	arc[start angle=\n4-\n9*90, delta angle={\n9*Mod(\n9*(\n1-\n4),360)}, radius=#2]
	-- cycle;
}

\newcommand{\fedge}[5][]{\fedgefull[fedge, #1]{\hedgeoffset}{#2}{#3}{#4}{#5}}

\pgfdeclarelayer{foreground}
\pgfsetlayers{main,foreground}
\newcommand{\hvertex}[1]{%
	\begin{pgfonlayer}{foreground}
		\fill #1 circle [radius=\hvertexradius];
	\end{pgfonlayer}}

\tikzset{xedge/.style={hedge, draw=red!60!yellow, fill=red!60!yellow}}

\newcommand{\gappsskeleton}{%
	\foreach \i/\a in {0/142,1/214,2/286,3/358,4/70}{\coordinate (x\i) at (\a:1.4cm);}%
	\foreach \i/\a in {5/106,6/178,7/250,8/322,9/34}{\coordinate (x\i) at (\a:2.5cm);}}
\newcommand{\gappsredblobs}{%
	\tedge{(x0)}{(x1)}{(x6)}\tedge{(x1)}{(x2)}{(x7)}\tedge{(x2)}{(x3)}{(x8)}%
	\tedge{(x3)}{(x4)}{(x9)}\tedge{(x4)}{(x0)}{(x5)}}
\newcommand{\gappsgraph}{%
	\draw[line width=1.25pt] (x0)--(x1)--(x2)--(x3)--(x4)--cycle;
	\draw[line width=1.25pt] (x4)--(x5)--(x0)--(x6)--(x1)--(x7)--(x2)--(x8)--(x3)--(x9)--cycle;
	\foreach \i in {0,...,9}{\hvertex{(x\i)}}}

\newcommand{\gappskfourskeleton}{%
	\coordinate (a) at (0,0);
	\coordinate (x) at ($(a)+(0:1.8)$);      \coordinate (c) at ($(x)+(0:1.8)$);
	\coordinate (b) at ($(a)+(-120:1.8)$);   \coordinate (y) at ($(x)+(-120:1.8)$);
	\coordinate (d) at ($(c)+(-120:1.8)$);
	\path ($(b)+(-0.2,-0.2)$) ($(c)+(0.2,0.2)$);}
\newcommand{\gappskfourgraph}{%
	\draw[line width=1.25pt] (a)--(b)--(y)--(d)--(c)--(x)--cycle;
	\draw[line width=1.25pt] (a)--(y) (x)--(y) (x)--(d);
	\foreach \p in {a,b,c,d,x,y}{\hvertex{(\p)}}}
\newcommand{\gappskfourlabels}{%
	\node[above=8pt] at (x) {$x$}; \node[below=8pt] at (y) {$y$};}

\usepackage{subcaption}
\def\ep{\textnormal{ep}}
\newcommand{\ttcopies}[2]{\binom{#1}{#2}_{\ep}}
\def\Fs{F_{\star}}

\begin{document}

\title[Unavoidable subgraphs in Ramsey graphs]{Unavoidable subgraphs in Ramsey graphs}

\author[Chr.~Reiher]{Christian Reiher}
\address{Fachbereich Mathematik, Universit\"at Hamburg, Hamburg, Germany}
\email{christian.reiher@uni-hamburg.de}

\author[V.~R\"{o}dl]{Vojt\v{e}ch R\"{o}dl}
\address{Department of Mathematics, Emory University, Atlanta, USA}
\email{vrodl@emory.edu}

\author[M.~Schacht]{Mathias Schacht}
\address{Fachbereich Mathematik, Universit\"at Hamburg, Hamburg, Germany}
\email{schacht@math.uni-hamburg.de}

\thanks{The second author is supported by NSF grant DMS~2300347.}


\begin{abstract}
	We study subgraphs that appear in large Ramsey graphs for a given graph~$F$.
	The recent girth Ramsey theorem of the first two authors asserts that there are Ramsey
	graphs such that all small subgraphs lie in `forests of copies of~$F$' amalgamated
	on vertices and edges. We derive a few further consequences from this structural result
	and investigate to what extent such forests of copies must be present in
	Ramsey graphs. As an application, we improve a recent result on generalised Sidon
	sets with Ramsey-type properties.
\end{abstract}

\maketitle
\setcounter{footnote}{1}

\section{Introduction}
\label{sec:introduction}
Some of the most fundamental and general types of questions in combinatorics
aim at understanding the interplay of \emph{local} and \emph{global} properties in discrete structures.
An insightful example for our discussion is provided by those local properties that are captured
by the appearance of subgraphs of fixed isomorphism type in large graphs with high chromatic number.
We regard a high chromatic number as a global property of graphs, since in general the chromatic
number cannot be bounded by studying subgraphs of given size.
In fact, this is a consequence of the work of Erd\H os~\cite{Er59}, which asserts the existence of
graphs with arbitrarily high chromatic number all of whose `small' subgraphs are forests.
This result is qualitatively complemented by the simple graph-theoretic fact that any graph of sufficiently high chromatic number
must contain all forests of fixed size.

Natural generalisations of these results, in the context of Ramsey theory for \emph{vertex colourings},
were obtained by Ne\v set\v ril and R\"odl~\cite{NR76} and by Diskin et al.~\cite{DHKZ24}. Here we investigate to what extent
these results generalise to \emph{edge colourings}. A central tool for our investigation is the
recent girth Ramsey theorem of the first two authors and we start with its formulation.

\subsection{Girth Ramsey theorem}
\label{sec:grt}
Given two graphs~$F$ and~$G$ and a number of colours~$r\geq 2$ we write $G\lra (F)_r$ for the statement that every~$r$-colouring
of the edge set of~$G$ yields a monochromatic copy of~$F$ which is induced in~$G$. If the monochromatic copy is not required to be induced, we write $G\nilra (F)_r$ instead. In both situations we also say that~$G$ is a \emph{Ramsey graph} for~$F$ and~$r$ colours.
If a graph~$G$ fails to be such a Ramsey graph, then we write~$G\nlra(F)_r$ and $G\ninlra(F)_r$, respectively.
It is easy to see that Ramsey graphs for a given graph~$F$ must contain copies of~$F$ sharing vertices and edges.
Roughly speaking, the girth Ramsey theorem (see Theorem~\ref{thm:grt} below) asserts that there are Ramsey graphs with the property that locally the copies
of~$F$ have a forest-like structure. In order to make this precise we define forests in this context.
\begin{dfn}[Forests of copies] \label{dfn:forcop}
	A set of graphs~$\ccF$ is called a \emph{forest of copies} if there
	exists an enumeration $\ccF=\{F_1, \dots, F_{|\ccF|}\}$ such that
	for every $j\in [2, |\ccF|]$ the set
	\[
		q_j=V(F_j)\cap\bigcup_{i<j}V(F_i)
	\]
	is either empty or a single vertex, or an edge belonging both to~$E(F_j)$
	and to $\bigcup_{i<j}E(F_i)$. Such an enumeration is called \emph{admissible}
	and we refer to~$q_j$ as the \emph{connector} of~$F_j$. Moreover, we refer to~$F_1$ as
	the \emph{initial copy} of the forest.
\end{dfn}
We remark that for a forest of copies~$\ccF$ of a graph~$F$ the graph $\fF=\bigcup \ccF$ may contain additional
copies of~$F$, which are not present in~$\ccF$ itself. It is also important to note that, contrary to forests in graphs,
subsets of forests of copies might not be such forests themselves, as the example in Figure~\ref{fig:12} illustrates.
On the left-hand side we see a `cycle' of triangles, i.e., a collection of
five triangles not forming a forest of copies. By adding two further edges, however,
we can hide this configuration inside a forest of eight triangles (see Figure~\ref{fig:12b}).
\begin{figure}[t]
	\centering
	\begin{subfigure}[b]{0.45\textwidth}
		\centering
		\begin{tikzpicture}[scale=1.2]
			\gappsskeleton
			\gappsredblobs
			\gappsgraph
		\end{tikzpicture}
		\caption{A `cycle' of triangles \dots}
		\label{fig:12a}
	\end{subfigure}
	\hfill
	\begin{subfigure}[b]{0.45\textwidth}
		\centering
		\begin{tikzpicture}[scale=1.2]
			\gappsskeleton
			\gappsredblobs
			\tedge[xedge]{(x0)}{(x3)}{(x2)}
			\tedge[xedge]{(x0)}{(x2)}{(x1)}
			\tedge[xedge]{(x0)}{(x4)}{(x3)}
			\gappsgraph
			\draw[line width=1.25pt] (x3)--(x0);
			\draw[line width=1.25pt] (x0)--(x2);
		\end{tikzpicture}
		\caption{\dots{} contained in a forest of triangles.}
		\label{fig:12b}
	\end{subfigure}
	\caption{Some subforests fail to be forests.}
	\label{fig:12}
\end{figure}

We denote by $\binom GF$ the set of all induced copies
of~$F$ in~$G$, i.e., the set of all induced subgraphs of~$G$ isomorphic to~$F$.
For a subsystem $\ccG\subseteq \binom GF$ and a number of colours~$r$ the
partition relation $\ccG\lra (F)_r$ indicates that for every~$r$-colouring
of $E(G)$ one
of the copies in~$\ccG$ is monochromatic.
With this notation at hand we can state the following version of the girth Ramsey theorem of the first two authors~\cite{RR}.
\begin{thm}[Girth Ramsey theorem] \label{thm:grt}
	Given a graph~$F$ and $r, \l\in \NN$ there exists a graph~$G$
	together with a system of copies $\ccG\subseteq\binom{G}{F}$
	satisfying not only $\ccG\lra (F)_r$, but also the following statement:
	For every $\ccF\subseteq \ccG$ with $|\ccF|\in [2, \l]$ there exists
	a set $\ccX\subseteq \ccG$ such that $|\ccX|\le |\ccF|-2$ and $\ccF\cup\ccX$
	is a forest of copies. \qed
\end{thm}
The r\^{o}le of~$\ccX$ in Theorem~\ref{thm:grt} is due to the fact that, as Figure~\ref{fig:12} shows,
$\ccF$ can only be demanded to be a subset of a forest of copies rather than an actual forest of copies.

Similarly, as Erd\H os' theorem concerns graphs~$G$ yielding monochromatic copies of~$K_2$ for vertex colourings,
Theorem~\ref{thm:grt} concerns graphs~$G$ yielding monochromatic copies of~$F$ for edge colourings. Both theorems then
assert that there are such graphs~$G$ with the additional property of containing locally only forests of copies of the target graph~$F$.

\subsection{New results}
Our first two results address the question to what extent the converse of Theorem~\ref{thm:grt} holds.
More precisely, we ask for which graphs~$F$ every forest of copies necessarily occurs in every Ramsey graph for sufficiently many colours.
In the positive direction we show that this holds for cycles and for balanced, complete, multipartite graphs.
\begin{thm}\label{thm:forcing}
	Let~$F$ be a cycle or a balanced, complete, multipartite graph with at least two edges. For every forest~$\ccF$ of copies of~$F$,
	there exists some $r\geq 2$ such that every Ramsey graph $G\nilra(F)_r$ contains a subgraph isomorphic
	to $\bigcup \ccF$.
\end{thm}
The proof of Theorem~\ref{thm:forcing} will be presented in \ssign\ref{sec:forcing}.
It seems plausible that Theorem~\ref{thm:forcing} extends to a larger family of graphs.
However, our next result shows that there are graphs for which such a statement fails.

\begin{thm}\label{thm:noforest}
	There exist infinitely many graphs~$F$ such that for every integer $r\geq 2$
	there exists a Ramsey graph $G\lra(F)_r$ with the following properties.
	\begin{enumerate}[label=\alabel]
		\item\label{it:14a} All copies of~$F$ in~$G$ are induced.
		\item\label{it:14b} Any two induced copies of~$F$ in~$G$
		      that share an edge $xz$ also share a vertex~$y$ that is neither
		      adjacent to~$x$ nor adjacent to~$z$ in~$G$.
	\end{enumerate}
\end{thm}

At first sight, property~\ref{it:14b} seems quite orthogonal to anything the girth
Ramsey theorem might be capable of doing for us. Nevertheless, the proof of
Theorem~\ref{thm:noforest} presented in \ssign\ref{sec:noforest} relies on a
version of the girth Ramsey theorem for ordered hypergraphs, which we discuss
in \ssign\ref{sec:ordered}, and on the fact, established in \ssign\ref{sec:nonew}, that
forests of copies of sufficiently connected hypergraphs contain no copies besides their
members. It might also be interesting to point out that the
graphs~$F$ constructed to this end are dense and highly connected, so the phenomenon
described here is not just some pathological degenerate case.

Our two further main results concern other local structures in Ramsey graphs. The next theorem
establishes the existence of Ramsey graphs for many colours
that locally fail to have the Ramsey property for two colours. In that sense it parallels the simple consequence of Erd\H os' theorem
that there are graphs of high chromatic number that are locally bipartite.
\begin{thm}\label{thm:simple}
	For every graph~$F$ containing a cycle and all integers $r\geq 2$ and $n\in\NN$ there exists a Ramsey graph~$G\lra(F)_r$ such that for every subgraph $H\subseteq G$ on at most~$n$ vertices we have $H\ninlra(F)_2$.
\end{thm}
We remark that the assumption that~$F$ contains a cycle cannot be dropped in Theorem~\ref{thm:simple}. For example, the conclusion obviously fails for matchings and stars.
We deduce Theorem~\ref{thm:simple} from the girth Ramsey theorem and we present the details of that proof in \ssign\ref{sec:simple}.

Our last main result concerns large free subgraphs of Ramsey graphs. As usual,
an~$F$-free graph contains no copy of~$F$, induced or not.
\begin{thm}\label{thm:free}
	Let~$F$ be a graph with at least two edges that cannot be disconnected by
	deleting either one vertex or two vertices forming an edge,
	and let~$\eps>0$.
	For every integer~$r\geq2$ there exists a Ramsey graph~$G\lra(F)_r$
	with $\girth(G)=\girth(F)$ and
	with the property that every subgraph~$H\subseteq G$ contains
	an~$F$-free subgraph~$H'\subseteq H$ satisfying
	\[
		|E(H')|
		\geq
		\Big(1-\frac{1}{|E(F)|}-\eps\Big)|E(H)|\,.
	\]
\end{thm}
Graphs~$F$ satisfying the assumption of Theorem~\ref{thm:free} must be~$2$-connected, but they are not
necessarily~$3$-connected. In the literature such graphs were sometimes called $2.5$-connected (see, e.g.,~\cite{NR78}).
However, in contrast to standard vertex connectivity, this concept is not closed under supergraphs.
In \ssign\ref{sec:nonew} we refer to such graphs as $(e, v)$-inseparable instead. Besides
all~$3$-connected graphs they include, e.g., all cycles and all complete bipartite
graphs~$K_{2,t}$ with $t\geq 2$.

We also remark that considering a subgraph $H=F$ shows that up to the~$\eps$,
the bound on~$|E(H')|$ in Theorem~\ref{thm:free} is best possible.
The constructive proof of Theorem~\ref{thm:free}, presented in \ssign\ref{sec:free},
relies on an edge-partitioned version of the girth Ramsey theorem for ordered graphs (see \ssign\ref{sec:epgrt}) and
employs uniform shift hypergraphs.

\begin{nrem}
	In the first version of the manuscript a weaker version of Theorem~\ref{thm:free} was
	stated. More precisely, the result was restricted to bipartite graphs~$F$ only and the
	bound on the number of edges in~$H'$ was weaker. While revising the manuscript we
	experimented with an LLM, which hinted at a proof of
	Proposition~\ref{cor:nonew} for graphs and at combining Ramsey's theorem for
	edge-partitioned graphs with high chromatic hypergraphs. These
	insights led to the current proof of Proposition~\ref{cor:nonew} detailed in \ssign\ref{sec:nonew}
	and the proof of Theorem~\ref{thm:free} presented in \ssign\ref{sec:free}.
\end{nrem}

\subsection{Ramsey infinite graphs}
It turns out that Theorem~\ref{thm:simple} also implies that graphs containing a cycle have infinitely many minimal
Ramsey graphs. We recall that a graph~$F$ is \emph{Ramsey infinite}
if there are infinitely many minimal graphs~$G$ satisfying the partition
relation $G\nilra(F)_2$ and otherwise we say~$F$ is \emph{Ramsey finite}.

Burr, Erd\H os, and Lov\'asz~\cite{BEL76} expressed the belief that with a few exceptions most graphs are Ramsey infinite
and this was confirmed through the work of Murty (see, e.g.,~\cite{BEL76}), of Burr, Erd\H os, and Lov\'asz~\cite{BEL76},
of Ne\v set\v ril and R\"odl~\cite{NR78}, of Faudree~\cite{Fa91}, and of R\"odl and Ruci\'nski (see, e.g., \L uczak~\cite{Lu94}).
As it turned out, the only Ramsey finite graphs are certain types of star forests, among which are
matchings of any size and stars with an odd number of edges (see the work of Faudree~\cite{Fa91} for a full classification).
The complementing case that all graphs containing a cycle are Ramsey infinite is a
consequence of the work of R\"odl and Ruci\'nski~\cites{RR93,RR95} on thresholds for Ramsey properties in random graphs. In \ssign\ref{sec:simple} we shall show that this follows easily from Theorem~\ref{thm:simple}.
\begin{thm}\label{thm:Rinf}
	Every graph containing a cycle is Ramsey infinite. \qed
\end{thm}

\subsection{Generalised Sidon sets}
Proposition~\ref{cor:free}, on which the proof of Theorem~\ref{thm:free} is based, also has
an application in additive number theory studied in~\cite{RRS26}. For a set $X\subseteq\NN$ and an
integer $k\ge 2$ we denote by
\[
	\rho_k(X)
	=
	\sup_{n\in\NN}\,\big|\big\{(x_1,\dots,x_k)\in X^k\colon x_1\le\dots\le x_k
	\tand x_1+\dots+x_k=n\big\}\big|
\]
the maximum number of representations of an integer as a sum of~$k$ terms from~$X$.

For an integer $h\ge 1$ the set~$X$ is called a \emph{$B_{k,h}$-set} if
$\rho_k(X)=h$, i.e., if some integer has~$h$ such representations, while no integer
has more of them. The~$B_{k,1}$-sets, all of whose~$k$-term sums are distinct, are
also called \emph{$B_k$-sets}, and~$B_2$-sets are usually referred to as
\emph{Sidon sets}.
In this context, we use the arrow notation
\[
	X\lra[k,h]_r
\]
for the statement that for every~$r$-colouring $X=X_1\dcup\dotsb\dcup X_r$ some colour
class~$X_q$ satisfies $\rho_k(X_q)\ge h$, i.e., some integer has~$h$ distinct
representations as a sum of~$k$ terms from~$X_q$.

Motivated by questions of Erd\H os and Newman (see, e.g.,~\cite{E80}) and of Alon and
Erd\H os~\cite{AE85}, it was shown in~\cite{RRS26}*{Theorem~1.1} that there
are~$B_{k,h}$-sets~$X$ satisfying $X\lra[k,h]_r$ for every integer $r\ge 2$ that are,
in addition, locally sparse in the sense that every finite
subset $Y\subseteq X$ contains a~$B_k$-set $Y'$ of size at least~$\frac{k-1}{2k}|Y|$.

Here we improve the density of $Y'$ in~$Y$ from $\frac{1}{2}-\frac{1}{2k}$ to
every $\mu<1-\frac{1}{2k}$, which is essentially optimal (see the discussion following Theorem~\ref{thm:sidon} below).
This is the case $h=2$ of the following result, which for every integer $h\geq 2$
yields subsets~$Y'$ with $\rho_k(Y')<h$ of every density $\mu<1-\frac{1}{kh}$.
\begin{thm}\label{thm:sidon}
	For all integers $k, h\ge 2$ and every real $\mu\in\bigl(0, 1-\frac1{kh}\bigr)$
	there exists an infinite~$B_{k,h}$-set $X\subseteq\NN$ with the following properties.
	\begin{enumerate}[label=\alabel]
		\item\label{it:sidon-ramsey} We have $X\lra[k,h]_r$ for every integer $r\ge 2$.
		\item\label{it:sidon-dense} Every finite subset $Y\subseteq X$ contains a
		      subset~$Y'$ with $|Y'|\ge\mu|Y|$ and $\rho_k(Y')<h$.
	\end{enumerate}
\end{thm}
The densities of $Y'$ obtained in Theorem~\ref{thm:sidon} are essentially optimal
as the following considerations show. Suppose that $h\ge 2$ and $X\subseteq\NN$ satisfies
$X\lra[k,h]_r$ for every $r\ge 2$. We consider the hypergraph~$\cA$ with vertex
set~$X$ whose edges are the sets
\[
	\big\{x^{(j)}_i\colon (i,j)\in[k]\times[h]\big\}
\]
formed
by~$h$ distinct representations $x^{(j)}_1+\dots+x^{(j)}_k$ with
$x^{(j)}_1\le\dots\le x^{(j)}_k$ of the same integer. Since the terms of these
representations need not be distinct, every edge of~$\cA$ has at most~$kh$ elements,
and owing to $h\ge 2$ it has at least two elements. A set $Y'\subseteq X$ satisfies
$\rho_k(Y')<h$ if and only if it contains no edge of~$\cA$, and so the partition
relations $X\lra[k,h]_r$ for every $r\geq 2$ assert that~$\cA$ has infinite chromatic number.
Consequently,~$\cA$ contains an infinite family of pairwise disjoint edges, as
otherwise the union~$T$ of a maximal family of pairwise disjoint edges would be a finite
set meeting every edge of~$\cA$, and colouring each vertex in~$T$ with its own colour
and~$X\setminus T$ with one further colour would leave no edge monochromatic.
Now the
union~$Y$ of~$t$ pairwise disjoint edges has at most~$tkh$ elements, and every subset
$Y'\subseteq Y$ with $\rho_k(Y')<h$ misses an element of each of these edges, whence
\[
	|Y'|
	\leq
	|Y|-t
	\leq
	|Y|-\frac{|Y|}{kh}
	=
	\Big(1-\frac1{kh}\Big)|Y|\,.
\]
Consequently, the upper bound on~$\mu$ imposed in Theorem~\ref{thm:sidon} is essentially optimal.

For $k=h=2$ this yields a sharp threshold for the question of Alon and
Erd\H os~\cite{AE85}, who asked whether a set $X\subseteq\NN$ is a finite union of Sidon
sets if, for some $\delta>0$, every finite subset $Y\subseteq X$ contains a Sidon set of
size at least~$\delta|Y|$. As~$X$ fails to be a finite union
of Sidon sets if and only if $X\lra[2,2]_r$ holds for every $r\ge 2$, the discussion
above shows that this implication is true for every $\delta>\frac34$, while
Theorem~\ref{thm:sidon} applied with $k=h=2$ shows that it fails for every
$\delta<\frac34$. Previously, it was shown in~\cite{NRS24} to fail for some $\delta>0$ and, subsequently, for
$\delta=\frac14$ in~\cite{RRS26}. We do not know whether it holds for $\delta=\frac34$.

We also remark that the set~$X$ constructed in the proof of
Theorem~\ref{thm:sidon} presented in~\ssign\ref{sec:sidon}
has all properties stated in~\cite{RRS26}*{Theorem~1.1}, since the proofs given there only rely on
the properties of the auxiliary graph established in Proposition~\ref{prop:sidon-graph}.

\section{Forests of copies of inseparable hypergraphs}
\label{sec:nonew}

A hypergraph is \emph{linear} if any pair of distinct edges shares at most one vertex.
In the context of linear hypergraphs, forests of copies can be defined in the
same way as for graphs, i.e., one just replaces the word `graphs' in the first line of
Definition~\ref{dfn:forcop} by `linear hypergraphs'. We show that forests
of copies of sufficiently connected linear hypergraphs contain no copies besides their
members. Here a \emph{copy} of a hypergraph~$F$ in a hypergraph~$H$ is a subhypergraph
of~$H$ isomorphic to~$F$, which does not need to be induced in~$H$.

Given a hypergraph~$H$ and subsets of vertices~$X$, $Y$ we denote by~$H[X]$ the
subhypergraph of~$H$ induced on~$X$ and we write~$H-Y$ for the
induced subhypergraph obtained from~$H$ by deleting the vertices in~$Y$ together with all
hyperedges containing a vertex from~$Y$. In particular, $H[X]=H-\overline{X}$ for the
complement $\overline{X}=V(H)\setminus X$. For simplicity we write~$H-v$ for $H-\{v\}$.
Finally, we recall that a hypergraph~$F$ is \emph{connected} if for
every partition $V(F)=X\dcup Y$ with $X, Y\ne\vn$ there is an edge $e\in E(F)$
such that $e\cap X\ne\vn$ and $e\cap Y\ne\vn$. The connectivity notion we require here is
rendered by the following definition, which generalises the property of graphs
required in Theorem~\ref{thm:free}.

\begin{dfn}[$(e,v)$-inseparable hypergraphs]\label{dfn:evins}
	We say a linear~$k$-uniform hypergraph~$F$ is \emph{$(e,v)$-inseparable} if it satisfies the following properties:
	\begin{enumerate}[label=\rmlabel]
		\item\label{it:evins-i} the hypergraph~$F$ is connected and contains at least two edges,
		\item\label{it:evins-ii} for every set $z\subseteq V(F)$ with $|z|<k$ such that no two vertices of~$z$ are contained in one edge, the hypergraph $F-z$ remains connected,
		\item\label{it:evins-iii} and for every edge $e\in E(F)$ the hypergraph $F-e$, obtained from~$F$ by removing
		      all vertices of~$e$, remains connected.
	\end{enumerate}
\end{dfn}

We remark that some hypergraphs that are not $(e,v)$-inseparable have forests of copies that contain
copies which are not members of the forest itself. For instance, the graph~$K^-_4$ obtained
from the complete graph on four vertices by deleting an edge satisfies
parts~\ref{it:evins-i} and~\ref{it:evins-ii}, while the deletion of the two vertices of the
edge joining its two vertices of degree three leaves two isolated vertices,
wherefore part~\ref{it:evins-iii} fails. In fact, Figure~\ref{fig:k4-} exhibits a forest of
two copies of~$K^-_4$ whose union contains an additional copy. In the other direction, the following
result tells us that for $(e,v)$-inseparable, linear hypergraphs such examples do not
exist.
\begin{figure}[t]
	\centering
	\begin{subfigure}[t]{0.45\textwidth}
		\centering
		\begin{tikzpicture}[scale=1.2]
			\gappskfourskeleton
			\fedge{(a)}{(b)}{(y)}{(x)}
			\fedge{(x)}{(y)}{(d)}{(c)}
			\gappskfourgraph
			\gappskfourlabels
		\end{tikzpicture}
		\caption{Two copies of~$K^-_4$ forming a forest~$\ccF$ with connector~$xy$ \dots}
		\label{fig:k4-a}
	\end{subfigure}
	\hfill
	\begin{subfigure}[t]{0.45\textwidth}
		\centering
		\begin{tikzpicture}[scale=1.2]
			\gappskfourskeleton
			\fedge{(a)}{(y)}{(d)}{(x)}
			\gappskfourgraph
			\gappskfourlabels
		\end{tikzpicture}
		\caption{\dots{} and a further copy contained in~$\bigcup\ccF$.}
		\label{fig:k4-b}
	\end{subfigure}
	\caption{Without inseparability new copies can arise.}
	\label{fig:k4-}
\end{figure}

\begin{prop}\label{cor:nonew}
	Let $k\geq 2$ and let~$\ccF$ be a forest of copies of an $(e,v)$-inseparable~$k$-uniform
	hypergraph~$F$. Then every copy of~$F$ in $\bigcup\ccF$ is a member of~$\ccF$.
\end{prop}

We also establish the following simple observation, where, in contrast to
Proposition~\ref{cor:nonew}, the hypergraph~$F$ is not required to be $(e,v)$-inseparable.
Instead, this property is imposed on the subhypergraph under consideration.
\begin{prop}\label{fact:indsub}
	Let $k\geq 2$, let~$\ccF$ be a forest of copies of a linear~$k$-uniform hypergraph~$F$,
	and let~$R$ be an induced subhypergraph of~$\bigcup\ccF$. If~$R$ is $(e,v)$-inseparable,
	then~$R$ is an induced subhypergraph of a member of~$\ccF$.
\end{prop}

We remark that the assumption of~$R$ being induced in~$\bigcup\ccF$ cannot be dropped, even
if~$F$ itself is $(e,v)$-inseparable. For instance, the triangles~$xya$ and~$xyb$ form a
forest of copies of the $(e,v)$-inseparable triangle, and its union contains the
$(e,v)$-inseparable cycle~$xayb$ of length four, which lies in no member of the forest,
but, owing to the chord~$xy$, fails to be induced. On the other hand, if~$F$ is
$(e,v)$-inseparable, then Proposition~\ref{fact:indsub} shows that every induced copy
of~$F$ in~$\bigcup\ccF$ is a member of~$\ccF$.
Applying the proposition to cycles~$R$ implies that nonempty forests of copies
of a graph preserve its girth.

\begin{cor}\label{cor:indcycle}
	Let~$\ccF$ be a nonempty forest of copies of a graph~$F$. Then every induced cycle
	of~$\bigcup\ccF$ is contained in a member of~$\ccF$ and, in particular,
	$\girth(\bigcup\ccF)=\girth(F)$.
\end{cor}

\begin{proof}
	As cycles are $(e,v)$-inseparable and shortest cycles are induced, this follows from
	Proposition~\ref{fact:indsub}.
\end{proof}

In the arguments that follow it will be convenient to employ the following notation.
Whenever we say that $\ccF=\{F_1,\dots,F_t\}$ is a forest of copies, we tacitly impose
that the enumeration is admissible, and we write $\fF_j=\bigcup_{i\in[j]}F_i$
and $\fF=\fF_t$, so that the connector of~$F_j$ is $q_j=V(F_j)\cap V(\fF_{j-1})$, and it
is empty, a single vertex, or an edge belonging to both~$F_j$ and~$\fF_{j-1}$. For
convenience we let~$\fF_0$ be the empty hypergraph with $V(\fF_0)=\vn$ and we
set $q_1=\vn$.

The proof of Proposition~\ref{cor:nonew} is prepared in three steps.
In \ssign\ref{sec:elementary-linear} we collect properties of forests of copies of
linear hypergraphs. In particular, Facts~\ref{fact:forest} and~\ref{fact:reroot} tell us that the
members of such a forest are induced in their union and that the initial copy of an admissible
enumeration may be chosen arbitrarily. Next
in \ssign\ref{sec:elementary-ev} we establish consequences of $(e,v)$-inseparability leading to
Fact~\ref{fact:attach}, describing the way the components
of $\fF-V(F_i)$ are attached to the member~$F_i$. Then, in
\ssign\ref{sec:pieces} we investigate the pieces into which a hypothetical
copy of~$F$ that is not a member
of~$\ccF$ is cut by such a member. Finally,
in \ssign\ref{sec:pfnonew} we deduce
Proposition~\ref{cor:nonew}. Subsequently, we give a short proof of
Proposition~\ref{fact:indsub} based on Facts~\ref{fact:forest} and~\ref{fact:evins}.

\subsection{Properties of forests of linear hypergraphs}
\label{sec:elementary-linear}
The facts presented here only rely on the linearity of
the given hypergraph~$F$ and make no use of inseparability.
The first collects basic properties of the interaction of
the copies of a given forest.
\begin{fact}\label{fact:forest}
	Let $k\geq 2$ and let $\ccF=\{F_1,\dots,F_t\}$ be a forest of copies of a
	linear~$k$-uniform hypergraph~$F$. Then the following holds:
	\begin{enumerate}[label=\alabel]
		\item\label{it:f1} for $j\in[t]$ no edge of~$\fF_j$ contains vertices of
		      both $V(F_j)\setminus q_j$ and $V(\fF_{j-1})\setminus q_j$, and every edge
		      of~$\fF_j$ inside $V(\fF_{j-1})$ is an edge of~$\fF_{j-1}$, i.e., $\fF_{j-1}$ is induced in~$\fF$,
		\item\label{it:f2} for $i\in[t]$ every edge of~$\fF$ inside $V(F_i)$ is an edge
		      of~$F_i$, i.e.,~$F_i$ is induced in~$\fF$, and
		\item\label{it:f3} for distinct~$i$, $j\in[t]$ the set $V(F_i)\cap V(F_j)$ is empty,
		      a single vertex, or an edge of both~$F_i$ and~$F_j$ and, in particular,~$\fF$ is
		      linear.
	\end{enumerate}
\end{fact}

\begin{proof}
	Beginning with~\ref{it:f1} we observe that every edge of~$\fF_j$ belongs either to~$F_j$
	or to~$\fF_{j-1}$ or to both. In the last case this edge coincides with~$q_j$ and we are done.
	Applying the second assertion successively to the indices $t, t-1, \dots, j$ we see that
	every edge of~$\fF$ inside~$V(\fF_{j-1})$ is an edge of~$\fF_{j-1}$, i.e., $\fF_{j-1}$ is
	induced in~$\fF$.

	For assertion~\ref{it:f2} let $f\subseteq V(F_i)$ be an edge of~$\fF$, which is not an edge of~$F_i$.
	The definition of a forest of copies yields $f\not\in E(\fF_{i-1})$ and, consequently, there is a
	minimal $j>i$ such that $f\in E(F_j)$ and $f\not\in E(\fF_{j-1})$. However, in this case we have
	$f=q_j$ contradicting $f\not\in E(\fF_{j-1})$.

	For property~\ref{it:f3} let $i<j$. If $|q_j|\leq 1$, then the first part is a simple consequence of
	\[
		V(F_i)\cap V(F_j)\subseteq V(\fF_{j-1})\cap V(F_j)\,.
	\]
	In the remaining case~$q_j$ is an edge and we may assume $q_j\in E(F_j)\setminus E(F_i)$.
	A simple induction on $\max\{i,j\}$ concludes the argument by
	replacing~$F_j$ with~$F_{j'}$ for the largest index $j'<j$ such that $q_j\in E(F_{j'})$.

	Regarding the linearity of~$\fF$ a simple induction on~$t$ suffices, since two edges
	of~$\fF$ sharing two vertices either belong both to~$F_t$ or both to~$\fF_{t-1}$, or
	their intersection lies in~$q_t$, in which case both of them coincide with the
	edge~$q_t$.
\end{proof}

Just as the root of a graph-theoretic tree can be chosen arbitrarily,
the following fact asserts the same for the initial copy of a forest of copies.
This observation appeared already in~\cite{RR}*{Lemma~15.3}
and we include the short proof for completeness.

\begin{fact}\label{fact:reroot}
	Let $k\geq 2$ and let $\ccF=\{F_1,\dots,F_t\}$ be a forest of copies of a
	linear~$k$-uniform hypergraph~$F$. For every $i\in[t]$ there is an admissible
	enumeration of~$\ccF$ starting with~$F_i$.
\end{fact}

\begin{proof}
	We argue by induction on~$t$. If $i<t$, we appeal to the induction assumption
	for~$\{F_1,\dots,F_{t-1}\}$ and append~$F_t$, whose connector is still~$q_t$. So
	let $i=t$. If $q_t=\vn$, then~$F_t$ is vertex-disjoint from~$\fF_{t-1}$
	and $F_t,F_1,\dots,F_{t-1}$ is as required. Otherwise~$q_t$ is a single vertex or an
	edge of~$\fF_{t-1}$, wherefore there is some $j<t$ with $q_t\subseteq V(F_j)$ and
	with $q_t\in E(F_j)$ in the latter case. Due to the induction assumption the
	family $\{F_1,\dots,F_{t-1}\}$ has an admissible enumeration
	$F_j=F'_1,F'_2,\dots,F'_{t-1}$ and we contend that $F_t,F'_1,\dots,F'_{t-1}$ is an
	admissible enumeration of~$\ccF$. In fact, the connector of~$F'_1$ becomes
	$V(F'_1)\cap V(F_t)=q_t$, which is a single vertex or an edge belonging to
	both~$F'_1$ and~$F_t$, while for $\l\in[2,t-1]$ the
	inclusion $V(F'_\l)\cap V(F_t)\subseteq q_t\subseteq V(F'_1)$ leaves the connector
	of~$F'_\l$ unaltered.
\end{proof}

For a forest~$\ccF$ of copies of a cycle, the last fact of this section asserts that every
cycle of~$\bigcup\ccF$ not contained in a member of~$\ccF$ shares at least three vertices
with each of two distinct members. We state it more generally for linear~$k$-uniform
hypergraphs and their subhypergraphs~$R$ with minimum vertex degree at least two.
\begin{fact}\label{fact:twohosts}
	Let $k\geq 2$ and let $\ccF=\{F_1,\dots,F_t\}$ be a forest of copies of a
	linear~$k$-uniform hypergraph~$F$ and let $R\subseteq\fF$ be a subhypergraph with
	minimum vertex degree $\delta(R)\geq 2$.

	If~$R$ is not contained in a single member
	of~$\ccF$, then at least two members of~$\ccF$ contain at least $k+1$ vertices of~$R$.
\end{fact}

\begin{proof}
	We note that for every admissible enumeration of~$\ccF$ the least index~$j$
	with $R\subseteq\fF_j$ satisfies $j\ge2$ and that~$F_j$ contains at least $k+1$
	vertices of~$R$.

	Indeed, for $j=1$ we would have~$R\subseteq \fF_1=F_1$, which is excluded by the assumption of the fact.
	Moreover, since~$\fF_{j-1}$ is induced in~$\fF$, the minimality of~$j$ yields a vertex
	$v\in V(R)\setminus V(\fF_{j-1})$, and $\delta(R)\geq 2$ combined
	with Fact~\ref{fact:forest}\ref{it:f1} leads to at least $k+1$ vertices of~$R$ in $V(F_j)$.

	Applying the same reasoning again to an admissible
	enumeration $F_j=F'_1,F'_2,\dots,F'_t$ provided by Fact~\ref{fact:reroot}
	yields an index $\l\ge2$ such that also~$F'_\l\neq F_j$ contains at least $k+1$ vertices of~$R$.
\end{proof}

\subsection{Properties of forests of inseparable hypergraphs}
\label{sec:elementary-ev}
The three facts of this section do rely on $(e,v)$-inseparability.
Fact~\ref{fact:evins} concerns a single copy of~$F$, whereas Facts~\ref{fact:conn}
and~\ref{fact:attach} describe how a member of a forest of copies of~$F$ is attached to
the remainder of the forest.
It will be convenient to appeal to the following consequences of $(e,v)$-inseparability.
\begin{fact}\label{fact:evins}
	Let $k\geq 2$ and let~$F$ be an $(e,v)$-inseparable~$k$-uniform hypergraph. Then~$F$ has
	more than~$k$ vertices, it has minimum vertex degree $\delta(F)\geq 2$, and $F-z$ is
	connected for every set~$z$ contained in an edge of~$F$.

	Moreover, if~$F$ is a
	subhypergraph of an arbitrary linear~$k$-uniform hypergraph~$H$ and $F-z$ is disconnected
	for some set $z\subseteq V(F)$ contained in an edge~$f$ of~$H$, then $z=f$ and we have $f\notin E(F)$.
\end{fact}

\begin{proof}
	As~$F$ is linear with at least two edges, it has more than~$k$ vertices, and being
	connected it has no isolated vertices. If a vertex~$w$ were contained in only one
	edge~$e$, then for $u\in e\setminus\{w\}$ the
	vertex~$w$ would be isolated in $F-u$, which violates part~\ref{it:evins-ii} of
	Definition~\ref{dfn:evins}.

	Now let $z\subseteq e\in E(F)$ be given. Since $F-e$ is connected by
	Definition~\ref{dfn:evins}\ref{it:evins-iii}, we may assume $z\subsetneq e$. Every
	vertex $v\in e\setminus z$ lies in an edge $e'\ne e$, which by linearity meets~$e$
	only in~$v$. Consequently,~$e'$ is an edge of $F-z$ joining~$v$ to $F-e$ and,
	as $V(F-z)$ is the union of $V(F-e)$ and $e\setminus z$, one checks immediately
	that~$F-z$ is connected.

	For the last assertion we note that, by what we have just shown,~$z$ is contained in no
	edge of~$F$. In particular, we have $f\notin E(F)$ and, since by the linearity of~$H$ an
	edge of~$F$ containing two vertices of $z\subseteq f$ would coincide with~$f$, no two
	vertices of~$z$ lie in a common edge of~$F$. Now
	Definition~\ref{dfn:evins}\ref{it:evins-ii} yields $|z|\ge k$, i.e., $z=f$.
\end{proof}

Facts~\ref{fact:forest} and~\ref{fact:evins} entail the following consequence for the
members of a forest of copies.

\begin{fact}\label{fact:conn}
	Let $k\geq 2$ and let $\ccF=\{F_1,\dots,F_t\}$ be a forest of copies of an $(e,v)$-inseparable
	$k$-uniform hypergraph~$F$. For all distinct~$i$, $j\in[t]$ the
	hypergraph $F_j-V(F_i)$ is connected and nonempty.
\end{fact}

\begin{proof}
	Owing to Fact~\ref{fact:forest}\ref{it:f3} the set $z=V(F_j)\cap V(F_i)$ is empty, a
	single vertex, or an edge of~$F_j$. In particular,~$z$ is contained in an edge of~$F_j$
	and, hence, Fact~\ref{fact:evins} tells us that $F_j-V(F_i)$ is connected. It is
	nonempty, since Fact~\ref{fact:evins} also yields $|V(F_j)|>k$, whereas $|z|\le k$.
\end{proof}

The last fact of this section tells us that a component of $\fF-V(F_i)$ can be
attached to~$F_i$ along a single edge only.
For a subhypergraph $C\subseteq H$ we write~$N_H(C)$ for the set of those
vertices~$x$ in $V(H)\setminus V(C)$ that are contained in an edge~$e$ of~$H$
with $e\cap V(C)\ne\vn$.

\begin{fact}\label{fact:attach}
	Let $k\geq 2$ and let $\ccF=\{F_1,\dots,F_t\}$ be a forest of copies of an $(e,v)$-inseparable
	$k$-uniform hypergraph~$F$. For every $i\in[t]$ and every component~$C$
	of $\fF-V(F_i)$ there is an edge $f\in E(F_i)$ with $N_\fF(C)\subseteq f$.
\end{fact}

\begin{proof}
	Since the assertion refers to~$\ccF$ and~$F_i$ only, Fact~\ref{fact:reroot} allows us
	to assume $i=1$ and we argue by induction on~$t$. We note that the case $t=1$ is void and so
	let $t\ge2$.

	We observe that $V(F_t)\cap V(F_1)\subseteq q_t$ and recall that, owing to
	Fact~\ref{fact:forest}\ref{it:f1}, every edge of~$\fF$ intersecting $V(F_t)\setminus q_t$ is
	contained in~$V(F_t)$, whereas every other edge of~$\fF$ is an edge of~$\fF_{t-1}$.
	In the remainder we consider two cases depending on the relationship of~$q_t$ and~$F_1$.

	Suppose first that $q_t\subseteq V(F_1)$. Then $V(F_t)\cap V(F_1)=q_t$. In this case, the
	assertion of Fact~\ref{fact:attach} holds for all components of $\fF_{t-1}-V(F_1)$
	by the induction assumption. Moreover, Fact~\ref{fact:conn} asserts that
	$C=F_t-V(F_1)$ is connected and nonempty. Now Fact~\ref{fact:forest}\ref{it:f1} tells us
	$N_\fF(C)\subseteq q_t$, and the
	connector~$q_t$ is empty, a single vertex of~$F_1$, or an edge of~$F_1$ by
	Fact~\ref{fact:forest}\ref{it:f2}. In all three cases
	Definition~\ref{dfn:evins}\ref{it:evins-i} and Fact~\ref{fact:evins} place~$q_t$
	inside an edge of~$F_1$.

	In the second case $q_t\not\subseteq V(F_1)$, which rules out $q_t=\vn$. So~$q_t$
	is a single vertex or an edge of~$\fF_{t-1}$ and, in particular, $q_t\subseteq V(F_j)$ for some~$j$ with $1<j<t$.
	Thus the
	nonempty set $q_t\setminus V(F_1)$ lies in $F_j-V(F_1)$,
	which Fact~\ref{fact:conn} reveals to be connected, and, hence, it is contained
	in a single component~$C_0$ of $\fF_{t-1}-V(F_1)$. As $F_t-V(F_1)$ is connected as well and meets
	$\fF_{t-1}-V(F_1)$ exactly in $q_t\setminus V(F_1)\subseteq V(C_0)$, the components
	of $\fF-V(F_1)$ are $C=C_0\cup(F_t-V(F_1))$ and those of $\fF_{t-1}-V(F_1)$ other
	than~$C_0$. For those latter components again the induction assumption applied to~$\fF_{t-1}$
	concludes the proof and it remains to consider $C=C_0\cup(F_t-V(F_1))$.

	For that we consider the edge $f\in E(F_1)$ with $N_{\fF_{t-1}}(C_0)\subseteq f$ that the
	induction assumption provides for the component~$C_0$ of $\fF_{t-1}-V(F_1)$, and we
	contend that $N_\fF(C)\subseteq f$. First, if~$q_t$ meets~$V(F_1)$,
	then $q_t\not\subseteq V(F_1)$ tells us that~$q_t$ is an edge of~$\fF_{t-1}$
	intersecting~$V(C_0)$, whence $q_t\cap V(F_1)\subseteq N_{\fF_{t-1}}(C_0)\subseteq f$.
	Now every edge~$e\in E(\fF)$ intersecting~$V(C)$ either contains a vertex from
	$V(F_t)\setminus q_t$ or it is an edge of~$\fF_{t-1}$. In the former
	case $e\subseteq V(F_t)$ and, owing to $V(F_t)\setminus V(F_1)\subseteq V(C)$, we
	obtain $e\setminus V(C)\subseteq q_t\cap V(F_1)\subseteq f$. In the latter
	case $e\cap V(F_t)\subseteq q_t$ discloses that~$e$ intersects~$V(C)$ in~$V(C_0)$ only,
	so that $e\setminus V(C)\subseteq N_{\fF_{t-1}}(C_0)\subseteq f$.
\end{proof}

\subsection{Pieces of a new copy}
\label{sec:pieces}
Suppose that the union of a forest of copies contains a copy of~$F$ which is not a member
of the forest. Here we study how the members of the forest cut such a copy into `pieces',
the outcome being Lemma~\ref{lem:piece}, on which the proof of
Proposition~\ref{cor:nonew} in \ssign\ref{sec:pfnonew} will be based.

Throughout this section $\ccF=\{F_1,\dots,F_t\}$ is a forest of copies of an
$(e,v)$-inseparable~$k$-uniform hypergraph~$F$ and $F_\star\subseteq\fF$ is a copy
of~$F$. For $j\in[t]$ we set $I_j=V(F_\star)\cap V(F_j)$, we call~$F_j$ a \emph{host}
of~$F_\star$ if $|I_j|\ge k+1$, and in this case the induced
subhypergraph $P_j=F_j[I_j]$ is the \emph{piece} of~$F_\star$ in~$F_j$. Owing to
Fact~\ref{fact:forest}\ref{it:f3} distinct hosts share at most~$k$ vertices of~$F_\star$,
so~$F_\star$ has at most $\binom{|V(F)|}{k+1}$ hosts, and
Fact~\ref{fact:forest}\ref{it:f2} tells us that the piece $P_j=\fF[I_j]$ contains every
edge of~$F_\star$ inside~$I_j$. If $I_j=V(F_\star)$ for some~$j$,
then $F_\star\subseteq\fF[V(F_j)]=F_j$ and $|E(F_\star)|=|E(F_j)|$
yield $F_\star=F_j\in\ccF$. In particular, a member of~$\ccF$ regarded as a copy of~$F$
has exactly one host, namely itself, whereas a copy with at least two hosts is not a
member.

In the remainder of this section we assume that~$F_\star$ is not a member of~$\ccF$,
so that $I_j\ne V(F_\star)$ for all $j\in[t]$. Since Fact~\ref{fact:evins} tells us that
every vertex of~$F_\star$ lies in at least two edges, while~$F_\star$ is not contained in a single
member of~$\ccF$, Fact~\ref{fact:twohosts} shows that~$F_\star$ has at least two hosts.

Roughly speaking, the following lemma informs us that for every host~$F_i$
every component of $\fF-V(F_i)$
containing a vertex of~$F_\star$ is attached to the host~$F_i$ along a single edge
of~$F_i$ which itself fails to be an edge of~$F_\star$, and that pieces inherit
the minimum vertex degree condition from~$F$.
\begin{lemma}\label{lem:piece}
	Let $k\geq 2$ and let $\ccF=\{F_1,\dots,F_t\}$ be a forest of copies of an $(e,v)$-inseparable
	$k$-uniform hypergraph~$F$, let~$F_\star\subseteq\fF$ be a copy of~$F$ that is not a
	member of~$\ccF$, and let~$F_i$ be a host of~$F_\star$. Then
	\begin{enumerate}[label=\alabel]
		\item\label{it:p1} for every component~$C$ of $\fF-V(F_i)$ containing a vertex
		      of~$F_\star$ there is an edge $e_C\in E(F_i)$ with $N_\fF(C)\subseteq e_C\subseteq I_i$
		      and $e_C\notin E(F_\star)$,
		\item\label{it:p2} and $\delta(P_i)\geq 2$.
	\end{enumerate}
\end{lemma}

\begin{proof}
	Beginning with~\ref{it:p1} we consider a component~$C$ of $\fF-V(F_i)$ containing a
	vertex of~$F_\star$. Fact~\ref{fact:attach} provides an edge $e_C\in E(F_i)$
	with $N_\fF(C)\subseteq e_C$. Every
	edge of~$F_\star$ meeting~$C$ is contained in $V(C)\cup e_C$, and $I_i\setminus e_C\ne\vn$
	because of $|I_i|\ge k+1$. Hence, for $z=e_C\cap V(F_\star)$, the vertices of~$F_\star$
	in~$C$ form a union of components of $F_\star-z$ avoiding the nonempty
	set $I_i\setminus e_C$, so $F_\star-z$ is disconnected. As~$\fF$ is linear by
	Fact~\ref{fact:forest}\ref{it:f3}, Fact~\ref{fact:evins} yields $z=e_C$ and $e_C\notin E(F_\star)$,
	whence $N_\fF(C)\subseteq e_C\subseteq I_i$.

	For part~\ref{it:p2} we appeal to
	Fact~\ref{fact:reroot} and assume $i=1$ for a simpler notation. We fix
	for every component~$C$ of $\fF-V(F_1)$ containing a vertex of~$F_\star$ an edge~$e_C$
	as in~\ref{it:p1} and we consider a vertex $w\in I_1$. Every
	edge of~$F_\star$ containing~$w$ either lies inside~$I_1$, in which case
	Fact~\ref{fact:forest}\ref{it:f2} tells us that it is an edge of~$P_1$, or it meets such
	a component~$C$, in which case $w\in N_\fF(C)\subseteq e_C\subseteq I_1$ exhibits the
	edge~$e_C$ of~$P_1$ containing~$w$. Edges of the first kind belong to~$F_\star$ and
	edges of the second kind do not, while Fact~\ref{fact:evins} places~$w$ in at least two
	edges of~$F_\star$. Thus we obtain two distinct edges of~$P_1$ containing~$w$ unless no
	edge of~$F_\star$ containing~$w$ is inside~$I_1$ and all components met by these edges
	lead to one and the same edge~$e'$ of~$P_1$, which then fails to be an edge of~$F_\star$.

	In order to exclude this case we recall that every edge of~$F_\star$ meeting a
	component~$C$ with $e_C=e'$ is contained in $V(C)\cup e'$. So in the
	hypergraph $F_\star-(e'\setminus\{w\})$ the vertex~$w$ together with the vertices
	of~$F_\star$ in the components~$C$ with $e_C=e'$ forms a union of components avoiding
	the nonempty set $I_1\setminus e'$, i.e., this hypergraph is disconnected. As
	$e'\setminus\{w\}\subseteq I_1$ is a proper subset of the edge~$e'$ of the linear
	hypergraph~$\fF$, this contradicts Fact~\ref{fact:evins}.
\end{proof}

Let~$F_i$ be a host of~$F_\star$ and let us fix, for every component~$C$ of $\fF-V(F_i)$
containing a vertex of~$F_\star$, an edge~$e_C$ as provided by part~\ref{it:p1}. We
record two consequences. First, as $I_i\ne V(F_\star)$, some component of $\fF-V(F_i)$
contains a vertex of~$F_\star$, so~$P_i$ has an edge that is not an edge of~$F_\star$.
Consequently,
\begin{equation}\label{eq:piececount}
	|E(F_i)\cap E(F_\star)|
	<
	|E(P_i)|\,.
\end{equation}
Second, every edge of~$F_\star$ lies inside~$I_i$ or inside $V(C)\cup e_C$ for a
component~$C$ as in~\ref{it:p1}.

\subsection{Avoiding unwanted copies in forests of copies}
\label{sec:pfnonew}
This section is devoted to the proofs of Propositions~\ref{cor:nonew} and~\ref{fact:indsub}.

For Proposition~\ref{cor:nonew}, supposing for a contradiction
that it fails, we choose such a forest of copies and an unwanted copy in such a way that the number of
hosts is as large as possible and glue fresh copies of the entire forest
along suitable edges, thus arriving at a configuration with even more hosts.

\begin{proof}[Proof of Proposition~\ref{cor:nonew}]
	Assume for the sake of contradiction that Proposition~\ref{cor:nonew}
	fails for some $(e,v)$-inseparable linear~$k$-uniform hypergraph~$F$. Among all
	forests of copies~${\ccF=\{F_1,\dots,F_t\}}$ of~$F$ and all copies $F_\star\subseteq\bigcup\ccF$ of~$F$ not
	belonging to~$\ccF$ we choose a pair $(\ccF,F_\star)$ such that the number~$h$ of hosts
	of~$F_\star$ is maximal and then we fix
	a host~$F_i$ of~$F_\star$ such that the number of edges of the piece~$P_i$ is maximal.
	Since $F_\star\neq F_j$ for every $j\in [t]$, we infer for every piece $P_j=F_j[I_j]$ that it is not a subhypergraph of $F_\star$, i.e.,
	\begin{equation}\label{eq:piece-size}
		|E(P_j)|\overset{\eqref{eq:piececount}}{>} |E(F_j)\cap E(F_\star)|\,.
	\end{equation}

	After fixing an isomorphism $\sigma\colon F_i\lra F_\star$ we may locate a subhypergraph $R\subseteq F_\star$
	isomorphic to~$P_i$ by simply setting $R=\sigma(P_i)$. Moreover, we note $R\subseteq F_\star$ satisfies
	$|E(R)|=|E(P_i)|$, $|V(R)|=|I_i|\ge k+1$, and $\delta(R)\geq 2$ follows from
	Lemma~\ref{lem:piece}\ref{it:p2}.
	Depending on whether some member of~$\ccF$
	contains all vertices of~$R$ we distinguish two cases.

	\smallskip

	{\hskip2em \it First Case. We have $V(R)\subseteq V(F_j)$ for some $j\in[t]$.}

	\nobreak\smallskip

	Then $V(R)\subseteq I_j$, so~$F_j$ is a host, and every edge of~$R$ is an edge
	of~$F_\star$ inside~$I_j$. Consequently, inequality~\eqref{eq:piece-size} tells us
	$|E(P_i)|=|E(R)|\le|E(P_j)|-1$, which
	contradicts the choice of~$F_i$.

	\smallskip

	{\hskip2em \it Second Case. We have $V(R)\not\subseteq V(F_j)$ for all $j\in[t]$.}

	\nobreak\smallskip

	As in the proof of Lemma~\ref{lem:piece}\ref{it:p2} we appeal to Fact~\ref{fact:reroot}
	and assume $i=1$ for a simpler notation.
	Now Fact~\ref{fact:twohosts} provides two members~$F_a$ and~$F_b$ of~$\ccF$ each of
	which contains at least $k+1$ vertices of~$R$. We shall construct a
	forest $\ccF'\supseteq\ccF$ and a copy $F'\subseteq\bigcup\ccF'$ of~$F$ with more
	than~$h$ hosts, which contradicts the choice of $(\ccF,F_\star)$ and thereby completes
	the proof.

	For every component~$C$ of $\fF-V(F_1)$ containing a vertex of~$F_\star$ we fix an
	edge~$e_C$ as provided by Lemma~\ref{lem:piece}\ref{it:p1} and we let~$\ccE$ be the set
	of the edges arising in this way. For $e\in\ccE$ let~$B_e$ be the subhypergraph
	of~$F_\star$ induced on the union of~$e$ and the sets $V(C)\cap V(F_\star)$ over all
	components~$C$ with $e_C=e$. Then $V(B_e)\cap I_1=e$ and for $e\ne e'$ the
	hypergraphs~$B_e$ and~$B_{e'}$ intersect in $e\cap e'$ only, while the second
	consequence recorded at the end of \ssign\ref{sec:pieces} places every edge
	of~$F_\star$ inside~$I_1$ or inside one of the sets $V(B_e)$. We also note
	that $\sigma(e)\in E(F_\star)$ for every $e\in\ccE\subseteq E(F_1)$.

	For every edge $e\in\ccE$ we take a
	copy $\ccF^{(e)}=\{F^{(e)}_1,\dots,F^{(e)}_t\}$ of~$\ccF$ together with an
	isomorphism $\iota_e\colon\fF\lra\fF^{(e)}=\bigcup\ccF^{(e)}$
	mapping~$F_j$ onto~$F^{(e)}_j$ for every~$j$, such that $V(\fF^{(e)})\cap V(\fF)=\sigma(e)$
	with $\iota_e(v)=\sigma(v)$ for $v\in e$, while for $e\ne e'$ the
	hypergraphs~$\fF^{(e)}$ and~$\fF^{(e')}$ intersect in $\sigma(e)\cap\sigma(e')$ only. In
	other words, $\fF^{(e)}$ is a fresh copy of~$\fF$ glued to~$\fF$ along the
	edge~$\sigma(e)$, which is identified with the copy of~$e$.
	Set $\ccF'=\ccF\cup\bigcup_{e\in\ccE}\ccF^{(e)}$.

	The set~$\ccF'$ is a forest of copies of~$F$. Indeed, we enumerate
	first $F_1,\dots,F_t$ and then, for one $e\in\ccE$ after another, $F^{(e)}_1,\dots,F^{(e)}_t$.
	The connector of $F^{(e)}_1=\iota_e(F_1)$ is~$\sigma(e)$, an edge of $\iota_e(F_1)$ and
	of~$\fF$, and for $j\ge2$ the member~$F^{(e)}_j$ meets its predecessors
	outside~$\ccF^{(e)}$ only within $\sigma(e)\subseteq V(F^{(e)}_1)$, so that its connector
	is the image under~$\iota_e$ of the connector of~$F_j$.

	Now define $\theta\colon V(F_\star)\lra V\big(\bigcup\ccF'\big)$ by $\theta(v)=\sigma(v)$
	for $v\in I_1$ and $\theta(v)=\iota_e(v)$ for $v\in V(B_e)\setminus e$. As the sets~$I_1$
	and $V(B_e)\setminus e$ for $e\in\ccE$ partition $V(F_\star)$ and their
	images $V(R)\subseteq V(\fF)$ and $\iota_e(V(B_e)\setminus e)\subseteq V(\fF^{(e)})\setminus\sigma(e)$
	are pairwise disjoint, the map~$\theta$ is injective. An edge of~$F_\star$ inside~$I_1$
	is mapped by~$\sigma$ to an edge of $F_\star\subseteq\fF$, and an edge of~$F_\star$
	inside $V(B_e)$ is mapped to an edge of~$\fF^{(e)}$, because~$\theta$ agrees with~$\iota_e$
	on $V(B_e)$. Thus $F'=\theta(F_\star)$ is a copy of~$F$ in $\bigcup\ccF'$.

	It remains to count the hosts of~$F'$. Since $V(F')\cap V(\fF)=V(R)$, the members~$F_a$
	and~$F_b$ found above are hosts of~$F'$. Moreover, let~$F_m$ with $m\ne 1$ be a host
	of~$F_\star$. In view of Fact~\ref{fact:conn} the hypergraph~$F_m-V(F_1)$ is connected
	and nonempty, so $V(F_m)\setminus V(F_1)$ lies in one component~$C_m$ of $\fF-V(F_1)$,
	and~$C_m$ contains the set $I_m\setminus I_1$ of vertices of~$F_\star$, which is
	nonempty because of $|I_m|\ge k+1>|V(F_m)\cap V(F_1)|$. Every vertex
	of $V(F_m)\cap V(F_1)$ lies in two edges of~$F_m$, at most one of which is contained
	in $V(F_m)\cap V(F_1)$, and the other one contains a vertex
	of $F_m-V(F_1)\subseteq C_m$. Consequently, $V(F_m)\cap V(F_1)\subseteq N_\fF(C_m)\subseteq e$
	for $e=e_{C_m}$ and, therefore, $I_m\subseteq V(B_e)$
	and $\theta(I_m)=\iota_e(I_m)\subseteq V(\iota_e(F_m))$. In particular, the member $\iota_e(F_m)$
	of~$\ccF^{(e)}$ is a host of~$F'$. Distinct indices~$m$ lead to distinct hosts,
	because~$\iota_e$ is injective and the families~$\ccF^{(e)}$ for $e\in\ccE$ have no
	member in common, and none of these hosts is a member of~$\ccF$. Altogether~$F'$ has at
	least $(h-1)+2=h+1$ hosts. In particular,~$F'$ is not a member of~$\ccF'$ and the
	pair $(\ccF',F')$ contradicts the maximal choice of~$h$, which concludes
	the proof of Proposition~\ref{cor:nonew}.
\end{proof}

We conclude with the proof of Proposition~\ref{fact:indsub}, which only relies
on Facts~\ref{fact:forest} and~\ref{fact:evins}.

\begin{proof}[Proof of Proposition~\ref{fact:indsub}]
	Writing $\ccF=\{F_1,\dots,F_t\}$, we let $j\in[t]$ be the least index
	with $R\subseteq\fF_j$ and set $z=q_j\cap V(R)$. We
	first show that $R-z$ is connected. If $|z|\le 1$, then~$z$ is contained in an edge
	of~$R$, since $\delta(R)\ge 2$, and Fact~\ref{fact:evins} applies. Otherwise~$q_j$ is an
	edge of~$\fF$ containing~$z$, and as~$\fF$ is linear by Fact~\ref{fact:forest}\ref{it:f3},
	a disconnected hypergraph $R-z$ would, by Fact~\ref{fact:evins}, lead to $z=q_j\notin E(R)$,
	contrary to~$R$ being induced in~$\fF$.

	Now Fact~\ref{fact:forest}\ref{it:f1} asserts that no edge of~$\fF_j$, and hence no edge
	of~$R-z$, contains vertices of both $V(F_j)\setminus q_j$ and $V(\fF_{j-1})\setminus q_j$.
	As these two sets partition $V(\fF_j)\setminus q_j$, which contains $V(R-z)=V(R)\setminus q_j$,
	while $z\subseteq q_j$ lies in both~$V(F_j)$ and~$V(\fF_{j-1})$, the connectedness
	of~$R-z$ leads to $V(R)\subseteq V(F_j)$ or $V(R)\subseteq V(\fF_{j-1})$. Since $V(R)\subseteq V(\fF_{j-1})$ contradicts the
	minimal choice of~$j$, we
	have $V(R)\subseteq V(F_j)$. Owing to Fact~\ref{fact:forest}\ref{it:f2}
	the copy~$F_j$ is induced in~$\fF$ and since~$R$ is induced by assumption,
	we arrive at $R=\fF[V(R)]=F_j[V(R)]$.
\end{proof}

\section{Ordered hypergraphs}
\label{sec:ordered}

We shall need the following version of Theorem~\ref{thm:grt} for ordered, linear hypergraphs.
An \emph{ordered hypergraph} $H=(V,E,<)$ is a hypergraph $(V,E)$ equipped with a total
order~$<$ of its vertex set. All subhypergraphs of~$H$ inherit the ordering induced by~$H$.
For two ordered hypergraphs~$H$ and~$F$ a \emph{copy}
of~$F$ in~$H$ is a subhypergraph~$F'$ of~$H$ such that $(V(F'),E(F'))$ and $(V(F),E(F))$
are isomorphic, and we call such a copy \emph{ordered} if it displays the same ordering
as~$F$ itself.
Accordingly, $\binom{H}{F}_{\!<}$ stands for the set of all ordered,
induced copies of~$F$ in~$H$, we write $\binom{H}{F}$ for the set of all induced copies
of~$F$ in~$H$, which agrees with $\binom{(V(H),E(H))}{(V(F),E(F))}$, and
$\binom{H}{F}_{\!\textnormal{n.n.i.}}$ denotes the set of all, not necessarily induced, copies
of~$F$ in~$H$. All three sets consist of subhypergraphs of~$H$ and with a slight abuse of notation the inclusions
$\binom{H}{F}_{\!<}\subseteq\binom{H}{F}\subseteq\binom{H}{F}_{\!\textnormal{n.n.i.}}$
hold trivially.
The variant of the girth Ramsey theorem for ordered linear hypergraphs can be stated as follows (cf.~\cite{RR}*{Theorem~15.12}).

\begin{thm}[Girth Ramsey theorem for ordered hypergraphs]
	\label{thm:ogrt}
	Given an ordered linear~$k$-uniform hypergraph~$F$ and $r, \l\in \NN$ there exists an ordered
	linear~$k$-uniform hypergraph~$H$ together with a system of ordered copies $\ccH\subseteq\binom{H}{F}_{\!<}$
	satisfying not only $\ccH\lra (F)_r$, but also the following statement:
	For every set of ordered copies $\ccN\subseteq \ccH$ with $|\ccN|\in [2, \l]$ there exists
	a set $\ccX\subseteq \ccH$ such that $|\ccX|\le |\ccN|-2$ and $\ccN\cup\ccX$
	is a forest of copies. \qed
\end{thm}

We remark that the system~$\ccH$ in Theorem~\ref{thm:ogrt} is necessary for linear hypergraphs~$F$ that are not sufficiently connected,
since in this case the hypergraph~$H$ may contain some `unwanted' copies spread over
several copies from~$\ccH$. This situation can be excluded if~$F$ is $(e,v)$-inseparable
(see Definition~\ref{dfn:evins}). In fact, combining Theorem~\ref{thm:ogrt}
with Proposition~\ref{cor:nonew} yields a Ramsey hypergraph~$H$ containing no copies of~$F$
besides the members of~$\ccH$.

\begin{cor}\label{cor:ogrt}
	Given an ordered, linear, $(e,v)$-inseparable, $k$-uniform hypergraph~$F$ and integers $r, \l\in \NN$ there exists an ordered
	linear~$k$-uniform hypergraph~$H$ satisfying not only~$H\lra(F)_r$, but also the following statements:
	\begin{enumerate}[label=\alabel]
		\item\label{iti:ogrt-useful} Every edge of~$H$ is contained in some ordered copy from $\binom{H}{F}_{\!<}$.
		\item\label{iti:ogrt-allinduced} We have $\binom{H}{F}_{\!\textnormal{n.n.i.}}=\binom{H}{F}=\binom{H}{F}_{\!<}$,
		      i.e., every copy of~$F$ in~$H$ is induced and displays the same ordering in~$H$ as~$F$.
		\item\label{iti:ogrt-forest} For every set $\ccN\subseteq \binom{H}{F}_{\!<}$ with $|\ccN|\in [2, \l]$ there exists
		      a set $\ccX\subseteq \binom{H}{F}_{\!<}$ such that $|\ccX|\le |\ccN|-2$ and $\ccN\cup\ccX$
		      is a forest of copies.
	\end{enumerate}
\end{cor}

Compared to Theorem~\ref{thm:ogrt} the main new claim is part~\ref{iti:ogrt-allinduced},
and this is where the $(e,v)$-inseparability is used.
We would like to point out that already the case $\l=1$ of Corollary~\ref{cor:ogrt}
seems nontrivial to us, even though its clause~\ref{iti:ogrt-forest} is void.
Corollary~\ref{cor:ogrt} will be employed in the proof of Theorem~\ref{thm:noforest}.

\begin{proof}[Proof of Corollary~\ref{cor:ogrt}]
	Given $F=(V_F,E_F,<_F)$ and integers~$r$ and~$\l$ we may assume that~$\l\geq |E_F|$. We apply Theorem~\ref{thm:ogrt}
	and obtain $H=(V_H,E_H,<_H)$ and a system of ordered copies $\ccH\subseteq \binom{H}{F}_{\!<}$.
	Clearly, $H\lra(F)_r$ and
	without loss of generality we assume that $H=\bigcup \ccH$. This is because we may simply
	delete those edges of~$H$ that belong to no copy from~$\ccH$, which does not affect the Ramsey property.
	Consequently, the hypergraph~$H$ enjoys property~\ref{iti:ogrt-useful}.

	For property~\ref{iti:ogrt-allinduced} we consider an
	arbitrary $F_\star\in\binom{H}{F}_{\!\textnormal{n.n.i.}}$ and, since $H=\bigcup\ccH$, we fix
	for every edge~$e$ of~$F_\star$ an ordered copy of~$F$ from~$\ccH$ containing~$e$. If all these copies coincide, they form a
	forest $\ccF\subseteq\ccH$ with a single member, and otherwise there are at least two and at
	most $|E_F|\le\l$ of them, so Theorem~\ref{thm:ogrt} yields a forest of
	copies $\ccF\subseteq\ccH$ of~$F$ containing all of them. In both cases we
	have $F_\star\subseteq\bigcup\ccF$ and Proposition~\ref{cor:nonew}, applied to the
	underlying unordered hypergraphs, shows $F_\star\in\ccF\subseteq\ccH$.
	So $\binom{H}{F}_{\!\textnormal{n.n.i.}}\subseteq\ccH\subseteq\binom{H}{F}_{\!<}$ and, together with
	the inclusions noted above, this proves not only~\ref{iti:ogrt-allinduced} but
	also $\ccH=\binom{H}{F}_{\!<}$.

	In particular, property~\ref{iti:ogrt-forest} is the statement on forests of copies
	from Theorem~\ref{thm:ogrt}.
\end{proof}

\section{Forcing forests of cycles and of complete multipartite graphs}\label{sec:forcing}
In this section we derive Theorem~\ref{thm:forcing}. The proof for cycles and the proof for
balanced, complete, multipartite graphs are different and we start the discussion with the latter.

\subsection{Forests of copies of balanced, complete, multipartite graphs}
The proof relies on a theorem of de Bruijn and Erd\H{o}s on set mappings.
A \emph{set mapping} over a set~$X$ is defined to be a map $f\colon X\lra\powerset(X)$
with $x\not\in f(x)$ for every $x\in X$. Subsets $Y\subseteq X$ such
that $y\not\in f(y')$ holds for all $y, y'\in Y$ are called \emph{free}.
Let us quote~\cite{DE51}*{Theorem~3}.

\begin{thm}[de Bruijn and Erd\H{o}s]\label{thm:DBE}
	Let~$f$ be a set mapping over a (finite or infinite) set~$X$, and let~$k$ be a
	positive integer. If $|f(x)|\le k$ for every $x\in X$, then~$X$ can be partitioned
	into $2k+1$ free sets. \qed
\end{thm}

We call a graph~$F$ \emph{strongly edge-transitive} if for all ordered pairs
$(x, y), (x', y')\in V(F)^2$ with $xy, x'y'\in E(F)$ there exists an automorphism
of~$F$ sending~$x$ to~$x'$ and~$y$ to~$y'$. For instance, all balanced, complete,
multipartite graphs, Kneser graphs, and cycles are strongly edge-transitive. Given
a strongly edge-transitive graph~$F$ with at least one edge and a positive
integer~$m$ we denote the graph obtained by gluing~$m$ copies of~$F$ along an edge
by~$(m, F)$. For clarity we point out that this graph has $m\cdot(|V(F)|-2)+2$ vertices
and~$m\cdot(|E(F)|-1)+1$ edges. The edge where the gluing occurs is called the \emph{kernel}
of~$(m, F)$. In the special case~${m=1}$ every edge can be viewed as a kernel
of $F\cong (1, F)$.

\begin{lemma}\label{lem:2127}
	Let~$F$ be a balanced, complete, multipartite graph with at least two edges.
	If 	for some positive integer~$m$ a graph~$H$ has no subgraphs isomorphic
	to $(m, F)$, then
	\[
		H\ninlra (F)_{4m|V(F)|}\,.
	\]
\end{lemma}

\begin{proof}
	We define a set mapping~$f$ over $E(H)$ as follows.
	Given $e\in E(H)$ we set $f(e)=\vn$ if~$e$ does not belong to any n.n.i.\ copy
	of~$F$ in~$H$. Otherwise we let~$m_e$ be maximal such that~$H$ contains
	a n.n.i.\ copy~$M_e$ of $(m_e, F)$ with kernel~$e$.
	Since~$H$ is $(m, F)$-free, we know $m_e<m$.
	For every vertex~$z$ of~$M_e$ not belonging to~$e$ let $E(z,e)$ be the set of the edges of~$H$ containing~$z$ and one vertex of~$e$.
	We note $1\leq |E(z,e)|\leq 2$, as~$F$ is a complete, multipartite graph.
	Let $f(e)=\bigcup_z E(z,e)\subseteq E(H)$ be the set of at most $2(|V(M_e)|-2)=2m_e(|V(F)|-2)$ edges
	arising in this way. Owing to $|f(e)|<2m|V(F)|$ for every~$e\in E(H)$,
	Theorem~\ref{thm:DBE} yields a partition of $E(H)$ into $4m|V(F)|$ sets which are
	free with respect to~$f$. In other words, there is a
	colouring $h\colon E(H)\lra [4m|V(F)|]$ whose colour classes are free sets.

	It suffices to show that no n.n.i.\ copy of~$F$ is monochromatic with respect
	to~$h$.
	Assume for the sake of contradiction that $F_\star$ is a counterexample
	and let $e=xy\in E(F_\star)$ be arbitrary. By the maximal choice of~$M_e$ there is
	a vertex $z\in V(M_e)\cap V(F_\star)$ distinct from~$x$ and~$y$. Consequently,
	\[
		E(F_\star)\cap f(e)
		\supseteq E(F_\star)\cap E(z,xy)
		\neq \emptyset
	\]
	and, as the colour classes of~$h$
	are free, we arrive at a contradiction to $F_\star$ being
	monochromatic.
\end{proof}

\begin{proof}[Proof of Theorem~\ref{thm:forcing} for balanced, complete, multipartite graphs]
	Let us call a forest~$\ccF$ of copies of~$F$ \emph{edge-glued} if it has an
	admissible enumeration all of whose connectors are edges.
	It can easily be confirmed that for every forest~$\ccF$ of copies of~$F$ there
	is an edge-glued forest~$\ccF'$ of copies of~$F$ such that $\bigcup\ccF'$ has
	a subgraph isomorphic to $\bigcup\ccF$. In fact, this holds for some~$\ccF'$
	with $|\ccF'|=O(|\ccF|)$, but this bound is of no importance to what follows.
	In any case, this observation shows that it suffices to prove our theorem
	for edge-glued forests~$\ccF$.

	We handle this case by induction on $|\ccF|$. In the base case, $|\ccF|=1$,
	it suffices to take~$r=2$. Now consider an edge-glued forest
	$\ccF=\{F_1, \dots, F_m\}$ with $m\ge 2$ together with an admissible
	enumeration all of whose connectors are edges. Set $\ccF^-=\ccF\setminus\{F_m\}$,
	let $r^-$ be a natural number obtained by applying the induction assumption
	to~$\ccF^-$, and set $r=r^-+4m|V(F)|^2$. Given any graph~$H$ with $H\nilra (F)_r$
	we define~$H'$ to be the spanning subgraph of~$H$ whose edges are the kernels
	of n.n.i.\ copies of $(m|V(F)|, F)$ in~$H$.
	By Lemma~\ref{lem:2127} the spanning subgraph~$H''$ of~$H$
	with edge set $E(H)\setminus E(H')$ satisfies $H''\ninlra (F)_{4m|V(F)|^2}$ and
	thus we have $H'\nilra (F)_{r^-}$. Our choice of $r^-$ yields a n.n.i.\ copy of
	$\bigcup \ccF^-$ in~$H'$. Denote the edge of~$H'$ playing the r\^{o}le of~$q_m$
	in that copy by~$q$. Since~$q$ is the kernel of a n.n.i.\ copy of $(m|V(F)|, F)$
	in~$H$ and
	\[
		\big|V\bigl(\bigcup\ccF^-\bigr)\big|
		=(m-1)(|V(F)|-2)+2
		<m\cdot |V(F)|\,,
	\]
	there is a n.n.i.\ copy
	of~$F$ intersecting our copy of $\bigcup\ccF^-$ only in~$q$. This proves that~$H$
	contains indeed a n.n.i.\ copy of $\bigcup\ccF$.
\end{proof}

\subsection{Forests of copies of cycles} The proof of Theorem~\ref{thm:forcing} for cycles
follows the one for balanced, complete, multipartite graphs with the key lemma, Lemma~\ref{lem:2127},
replaced by the following statement.
\begin{lemma}\label{lem:cycleforcing}
	For all integers $\l\ge 3$, $m\ge 1$ there exists a constant~$r$ such that
	every graph~$H$ without subgraphs isomorphic to $(m, C_\l)$ satisfies
	$H\ninlra (C_\l)_r$.
\end{lemma}

\begin{proof}
	Let us call a number \emph{bounded} if it is $O_{\l, m}(1)$, i.e., if it
	has an upper bound in terms of~$\l$ and~$m$, which is independent of~$H$.
	Similarly, a set is \emph{bounded} if its cardinality is.

	We define recursively for every edge $xy$ of~$H$ an increasing
	chain~$(K_i(xy))_{i\in [\l]}$ of bounded vertex sets.
	In the first step, we choose~$m_{xy}$ maximal such that~$H$ contains
	a n.n.i.\ copy of $(m_{xy}, C_\l)$ with
	kernel $xy$ and let $K_1(xy)$ be the vertex set of such a copy. In the degenerate
	case $m_{xy}=0$, i.e., if no n.n.i.\ copy of $C_\l$ in~$H$ passes through $xy$,
	we simply set $K_1(xy)=\{x, y\}$. Notice that our
	assumption on~$H$ yields $m_{xy}<m$, whence $|K_1(xy)|\le 2+(\l-2)m$ is indeed
	bounded.

	Now suppose that for some positive $\lambda<\l$ the bounded
	set $K_\lambda(xy)\subseteq V(H)$ has already been defined.
	To every pair $p=(uv, k)\in K_\lambda(xy)^{(2)}\times [2, \l)$ we assign a
	bounded set~$B_p$ as follows. If~$H$ contains more than $2\l m$ internally
	vertex-disjoint paths of length~$k$ from~$u$ to~$v$ we put $B_p=\vn$.
	Otherwise we consider a
	maximal collection of such paths and let~$B_p$ denote the set of their inner
	vertices, so that $|B_p|\le 2k\l m<2\l^2 m$ is bounded. The number of pairs~$p$
	considered here is at most $|K_\lambda(xy)|^2\l$ and, therefore, bounded.
	Thus the union~$K_{\lambda+1}(xy)$ of $K_\lambda(xy)$ with all these vertex
	sets~$B_p$ is bounded as well.

	Having thereby completed our recursive definition we employ Theorem~\ref{thm:DBE}
	in order to get a colouring $h\colon E(H)\lra [r]$ with boundedly many colours
	such that for every $xy\in E(H)$ the only edge $e\subseteq K_\l(xy)$ with the
	same colour as $xy$ is $xy$ itself.

	\begin{clm}\label{clm:monopath}
		If $x_1\dots x_\l$ denotes a n.n.i.\ copy of $C_\l$ in~$H$ and
		\begin{equation}\label{eq:1721}
			K_1(x_1x_2)\cap\{x_1, \dots, x_\l\}\subseteq \{x_1, \dots, x_i\}
		\end{equation}
		for some $i\in [\l]$, then the path $x_1\dots x_i$ fails to be
		monochromatic with respect to~$h$.
	\end{clm}

	\begin{proof}
		Arguing indirectly we consider a counterexample such that~$i$ is as
		small as possible. Due to $x_1, x_2\in K_1(x_1x_2)$ we have $i\ge 2$.
		In fact, our maximal choice of $K_1(x_1x_2)$ guarantees that the left
		side of~\eqref{eq:1721} is larger than $\{x_1, x_2\}$, whence $i\ge 3$.
		Now $i=3$ would entail $x_3\in K_1(x_1x_2)$ and our choice of~$h$ would give rise
		to the contradiction $h(x_1x_2)\ne h(x_2x_3)$.

		For all these reasons we have $i\ge 4$. The minimality of~$i$
		applied to the cycle
		\[
			x_{i-1}\dots x_1x_\l\dots x_i
		\]
		yields
		$K_1(x_{i-1}x_{i-2})\cap \{x_1, \dots, x_\l\}\not\subseteq \{x_1, \dots, x_{i-1}\}$ and, consequently, there is some vertex in
		$K_1(x_{i-1}x_{i-2})\cap \{x_i, \dots, x_\l\}$.

		In particular, there is a pair $(j, \lambda)\in [i, \l]\times [\l]$
		such that $x_j\in K_\lambda(x_{i-1}x_{i-2})$ and $j+\lambda\le \l+1$.
		Indeed, we have just shown that even with $\lambda=1$ such a pair exists.
		For the remainder of the argument we fix $\lambda_{\star}\leq \l$ as large as possible
		such that there is an index $j_{\star}\in [i, \l]$ with
		\[
			x_{j_{\star}}\in K_{\lambda_{\star}}(x_{i-1}x_{i-2})
			\qqand
			j_{\star}+\lambda_{\star}\le \l+1\,.
		\]
		Due to $j_\star\ge i>1$ we have $\lambda_{\star}<\l$.

		If we had $j_{\star}=i$, then
		$x_i\in K_{\lambda_{\star}}(x_{i-1}x_{i-2})\subseteq K_\l(x_{i-1}x_{i-2})$
		would imply
		\[
			h(x_{i-1}x_{i-2})\ne h(x_{i-1}x_i)\,,
		\]
		contrary to the
		fact that the path $x_1\dots x_i$ is monochromatic with respect to~$h$.

		Consequently, we have $j_{\star}>i$ and for the definition of the set
		$K_{\lambda_{\star}+1}(x_{i-1}x_{i-2})$ the pair
		\[
			p=(x_{i-1}x_{j_{\star}}, j_{\star}-i+1)
		\]
		was considered.

		Suppose first that $B_p\ne\varnothing$, so that~$B_p$ needs to intersect
		$\{x_i, \dots, x_{j_{\star}-1}\}$ in some vertex~$x_{j'}$, say. But having
		$x_{j'}\in K_{\lambda_{\star}+1}(x_{i-1}x_{i-2})$
		and $j'+(\lambda_{\star}+1)\le j_{\star}+\lambda_{\star}\le \l+1$ is contrary to the maximality
		of~$\lambda_{\star}$.
		This contradiction proves $B_p=\varnothing$, which means that there are
		more than $2\l m$ internally vertex-disjoint paths of length~${j_{\star}-i+1}$
		from~$x_{i-1}$ to $x_{j_\star}$. At most $|K_1(x_1x_2)|\le\l m$ of these paths
		have an inner vertex in $K_1(x_1x_2)$, and at most~$\l$ have an
		inner vertex in $\{x_1, \dots, x_\l\}$. Consequently, there is a path
		$x_{i-1}y_i\dots y_{j_{\star}-1}x_{j_{\star}}$ such that
		\[
			x_1\dots x_{i-1}y_i\dots y_{j_{\star}-1} x_{j_{\star}}\dots x_\l
		\]
		is a cycle in~$H$ and $\{y_i, \dots, y_{j_{\star}-1}\}\cap K_1(x_1x_2)=\varnothing$.
		Clearly this cycle contradicts the minimal choice of~$i$ at the beginning
		of the proof of our claim.
	\end{proof}

	It remains to observe that the case $i=\l$ of Claim~\ref{clm:monopath} shows,
	in particular, that no~$\l$-cycle is monochromatic with respect to~$h$.
\end{proof}

Imitating the proof of Theorem~\ref{thm:forcing} for balanced, complete, multipartite graphs
with Lemma~\ref{lem:2127} replaced by Lemma~\ref{lem:cycleforcing}
yields Theorem~\ref{thm:forcing} for cycles and we omit the details.

\section{Ramsey graphs containing no forest of copies}\label{sec:noforest}
In this section we establish Theorem~\ref{thm:noforest}, which concerns the existence
of Ramsey graphs $G\lra(F)_r$ that only contain trivial forests of copies of~$F$. The proof naturally splits into two parts.
In \ssign\ref{sec:derivedF} we construct suitable graphs~$F$ and in \ssign\ref{sec:noforestG}
we apply the girth Ramsey theorem to obtain the corresponding Ramsey graphs~$G$.

\subsection{Graphs derived from ordered linear hypergraphs}
\label{sec:derivedF}
The graphs~$F$ will be derived from ordered, linear~$3$-uniform hypergraphs $S=(V,E,<)$.
For that, for an ordered hypergraph~$S$ we define the (unordered) graph $F_S=(V,E')$ on the same vertex set
by
\begin{equation}\label{eq:FS}
	xz\in E'
	\qquad
	\Longleftrightarrow
	\qquad
	\text{there is}\ xyz\in E\ \text{with}\ x<y<z\,.
\end{equation}
In other words, the edge set of the graph~$F_S$ is
obtained by removing the `intermediate vertices' from all edges of~$S$.
For the proof of Theorem~\ref{thm:noforest} we require that~$S$ is $(e,v)$-inseparable and~$F_S$
is~$4$-connected.

\begin{prop}\label{prop:3lin}
	There are infinitely many ordered, $(e,v)$-inseparable, linear~$3$-uniform hypergraphs~$S$ such that the
	graph~$F_S$ is~$4$-connected.
\end{prop}
\begin{proof}
	For any sufficiently large integer~$n$ divisible by~$16$, consider the ordered, linear~$3$-uniform hypergraph $S=(V,E,<)$
	defined on $V=[n]$ with its natural ordering and
	\[
		E=\big\{xyz\in[n]^{(3)}\colon x+y+z=n\ \text{or}\ x+y+z=2n\big\}\,.
	\]
	Clearly, $S$ is linear. Since for edges of~$S$ the vertices~$x$, $y$, and~$z$ must be distinct, one may observe that
	every fixed vertex~$x$ forms an edge in~$E$ with every other vertex~$b$, with at most three exceptions:
	the vertex $n-2x$ or $2n-2x$ (whichever lies in~$[n]$), for which the third vertex of the candidate
	edge would coincide with~$x$, and, if~$x$ is even, the two vertices $(n-x)/2$ and $(2n-x)/2$,
	for which it would coincide with~$b$. Consequently, for distinct vertices~$x$ and~$x'$
	there are at least $n-8$ choices for $b\in[n]\setminus \{x,x'\}$ for which there are vertices~$a$, $a'$ such that
	\[
		xab\in E\qand ba'x'\in E\,.
	\]
	Since any vertex can play each of the r\^oles~$a$, $b$, $a'$ at most once,
	the removal of any vertex different from~$x$ and~$x'$
	can destroy at most three such~$x$-$x'$ paths. Consequently,~$x$ and~$x'$
	cannot be separated in~$S$ by removing fewer than $(n-8)/3$ vertices and the $(e,v)$-inseparability of~$S$ follows.

	We prepare the proof of the~$4$-connectivity of~$F_S$ by analysing its minimum degree.
	This will later exclude the existence of `small' components.

	\begin{clm}\label{clm:minFS}
		We have $\delta(F_S)\geq n/9$.
	\end{clm}

	\begin{proof}
		We show that for every vertex $x\in [n]$
		there are at least $n/8-2\geq n/9$ pairs $\{a,b\}\in[n]^{(2)}$
		such that $xab\in E$ and either $x<\min\{a,b\}$ or $x>\max\{a,b\}$. In fact,
		one can check that depending on the choice of~$x$ those pairs $\{a,b\}$ can be
		found with at least one of the vertices~$a$ and~$b$ from the interval $I(x)$
		defined by
		\[
			I(x)=	\begin{cases}
				[n/4,n/2]\,,      & \text{if}\ x\in[1,n/4-1]    \\
				[3n/4,n]\,,       & \text{if}\ x\in[n/4,n/2-1]  \\
				[1,n/4]\,,        & \text{if}\ x\in[n/2,3n/4-1] \\
				[n/2+1,3n/4-1]\,, & \text{if}\ x\in[3n/4,n]\,.
			\end{cases}
		\]
		This proves the claim.
	\end{proof}

	The proof of the fact that~$F_S$ is sufficiently connected is rendered by the following claim.
	\begin{clm}
		The derived graph~$F_S$ is $n/100$-connected.
	\end{clm}

	\begin{proof}
		Let $U\subseteq [n]$ be an arbitrary set of at most $n/100$ vertices and set
		\[
			x_\star=\min\big([n]\setminus U\big)
			\qqand
			x^\star=\max\big([n]\setminus U\big)\,.
		\]
		We plan to show that these vertices are in the same component of $F_S-U$
		and that this component consists of `almost all' vertices. Together with
		the minimum degree property of~$F_S$ this will show that $F_S-U$ is indeed
		connected.

		Clearly, we have $x_{\star}\leq |U|+1\leq n/16$ and $x^\star\geq n-|U|\geq 15n/16$.
		Thus for every positive integer $k< n/12$ the edges
		\[
			\{x_\star, n/8+k-x_\star, 7n/8-k\}\,, \,\,\,
			\{n/2-k, 5n/8+2k, 7n/8-k\}\,, \,\,\,
			\{n/2-k, 3n/2+k-x^\star, x^\star\}
		\]
		of~$S$ yield a path $x_\star$-$(7n/8-k)$-$(n/2-k)$-$x^\star$ in~$F_S$.
		As these paths are internally vertex-disjoint, at least one of them is
		present in $F_S-U$ and, consequently, there is a component~$C$ of $F_S-U$
		containing both~$x_\star$ and~$x^\star$.

		Due to $x_\star<n/4$, for $k\in[n/2-2x_\star]$ the edges $\{x_\star,x_\star+k, n-2x_\star-k\}$ of~$S$ certify
		$N_{F_S}(x_\star)\supseteq [n/2,n-2x_\star-1]$, whence
		\[
			\big|[n/2, n]\setminus V(C)\big|
			\le 2x_\star+1+|U|
			\le 3|U|+3\,.
		\]
		Similarly, $x^\star>3n/4$ leads to $N_{F_S}(x^\star)\supseteq[2n-2x^\star+1,n/2-1]$,
		for which reason
		\[
			\big|[1, n/2-1]\setminus V(C)\big|
			\le 2n-2x^\star+|U|
			\le 3|U|\,.
		\]
		Altogether, we obtain $\big|[n]\setminus V(C)\big|\le 6|U|+3$. Combining the minimum degree $\delta(F_S)\geq n/9$ provided by Claim~\ref{clm:minFS}
		and $|U|\le n/100$, we infer that $F_S-U$ is indeed connected.
	\end{proof}
	This concludes the proof of Proposition~\ref{prop:3lin}.
\end{proof}

\subsection{Ramsey graphs for graphs derived from linear hypergraphs}
\label{sec:noforestG}
The proof of Theorem~\ref{thm:noforest} relies on Corollary~\ref{cor:ogrt} and Proposition~\ref{prop:3lin}.
\begin{proof}[Proof of Theorem~\ref{thm:noforest}]
	Let $S=(V_S,E_S,<_S)$ be an ordered $(e,v)$-inseparable linear hypergraph provided by Proposition~\ref{prop:3lin} and let
	$F=F_S$ be the derived~$4$-connected graph. For~${r\geq 2}$ let $H=(V_H,E_H,<_H)$ be the ordered, linear hypergraph provided by
	Corollary~\ref{cor:ogrt} applied to~$S$ with $\l=|E(F)|$.
	We derive the graph $G=F_H$ from~$H$
	in the same way as~$F$ was derived from~$S$, i.e.,
	\[
		xz\in E(G)
		\qquad\Longleftrightarrow\qquad
		\text{there is}\ xyz\in E(H)\ \text{with}\ x<y<z\,.
	\]
	\begingroup\sloppy
	Owing to the linearity of~$H$ this construction yields
	a one-to-one correspondence between~$E(H)$ and~$E(G)$.\par\endgroup

	\begin{clm}\label{clm:46}
		Every copy $S_\star\in \binom HS$ gives rise to an induced copy
		$F_{S_\star}\in \binom GF$.
	\end{clm}

	\begin{proof}
		By Corollary~\ref{cor:ogrt}\ref{iti:ogrt-allinduced} we know that $S_\star$
		appears with the correct ordering in~$H$ and, therefore,
		$F_\star=F_{S_\star}$ is a subgraph of~$G$ isomorphic to~$F$. We need
		to show that this subgraph is induced.

		Suppose to this end that $x, z\in V(F_\star)$ are connected by an
		edge $xz\in E(G)$. This means that for some vertex $y\in V(H)$
		with $x<y<z$ there is an edge $xyz\in E(H)$. Since~$H$ is linear,
		it has no other edges containing~$x$ and~$z$. We need to show that $xyz\in E(S_\star)$.

		By Corollary~\ref{cor:ogrt}\ref{iti:ogrt-useful} there is a
		copy $S_{\star\star}\in \binom HS$ containing the edge $xyz$.
		If $S_\star=S_{\star\star}$ we are done, so suppose that this is not
		the case. By Corollary~\ref{cor:ogrt}\ref{iti:ogrt-forest} the set
		$\{S_\star, S_{\star\star}\}$ is a forest of copies of~$S$. As the
		set $q=V(S_\star)\cap V(S_{\star\star})$ contains~$x$ and~$z$, this
		is only possible if $q=xyz$ is an edge of $S_\star$.
	\end{proof}

	This immediately implies that~$G$ enjoys the Ramsey property $G\lra(F)_r$.
	Indeed, any~$r$-colouring of~$E(G)$ defines an~$r$-colouring of~$E(H)$.
	Some copy $S_\star\in \binom HS$ needs to be monochromatic with respect
	to this induced colouring and by Claim~\ref{clm:46} we get a monochromatic
	copy $F_{S_\star}\in \binom GF$.

	\begin{clm}\label{clm:noforest}
		Every n.n.i.\ copy $F_\star$ of~$F$ in~$G$
		is of the form $F_{\star}=F_{S_\star}$ for some $S_\star\in\binom{H}{S}_{\!<}$.
	\end{clm}
	\begin{proof}
		For every edge $e=xz$ of $F_\star$ there exists a unique triple $T_e=xyz$ in $E(H)$ and by property~\ref{iti:ogrt-useful}
		of Corollary~\ref{cor:ogrt} we may fix for every $e\in E(F_\star)$ an ordered copy of~$S$ from~$\binom{H}{S}_{\!<}$ containing the triple~$T_e$.
		If all these at most $|E(F)|$ copies coincide, they form a forest of copies with a single member, and otherwise
		property~\ref{iti:ogrt-forest} of Corollary~\ref{cor:ogrt} applies to them. In both cases we obtain a forest
		of copies~$\ccS$ of~$S$ with an admissible enumeration $S_1,\dots,S_t$ for
		some $t\leq 2|E(F)|-2$ such that $\bigcup \ccS$ contains all triples~$T_e$
		for $e\in E(F_\star)$. We may assume that~$t$ is minimal with that property. If $t=1$ we are done, so suppose
		that $t\ge 2$, and let~$q$ be the connector of~$S_t$.
		Due to the minimal choice of~$t$, the copy~$S_t$ contains at least one of those triples~$T_e$ and $T_e\setminus q\neq\emptyset$.

		If $|q|\le 1$, then $T_e\setminus q$ contains at least one of the two vertices of~$e$
		and, as every edge of~$F_\star$ lies entirely in $V(S_t)$ or entirely in $\bigcup_{j<t}V(S_j)$,
		the set $q\cap V(F_\star)$ separates these two parts of~$V(F_\star)$. Since $F_\star$
		is~$4$-connected, this yields $V(F_\star)\subseteq V(S_t)$ and, hence, $V(F_\star)=V(S_t)$.
		Now the induced copy~$F_{S_t}$ provided by Claim~\ref{clm:46} has the same vertex set
		and contains~$F_\star$, so $F_\star=F_{S_t}$ and we are done.

		In the remaining case $|q|=3$ and $q\in E(H)$. Consequently, the linearity of~$H$ implies again $|T_e\cap q|\leq 1$
		and the same argument as in the case $|q|\le 1$ yields the claim.
	\end{proof}

	Both claims together establish part~\ref{it:14a} of Theorem~\ref{thm:noforest}.
	For part~\ref{it:14b} we consider two copies~$F_\star$,~$F_{\star\star}$
	in~$G$ sharing an edge $xz\in E(G)$
	with $x<z$. The definition of~$G$ yields a vertex~$y$ with $x<y<z$
	and $xyz\in E(H)$.
	Moreover, Claim~\ref{clm:noforest} implies
	$xyz\in E(S_\star)\cap E(S_{\star\star})$ and, hence,~$y$ is a common
	vertex of~$F_\star$,~$F_{\star\star}$. Finally, $xy, yz\not\in E(G)$
	follows from the linearity of~$H$.
\end{proof}

\section{Globality of Ramsey properties and Ramsey infinite graphs}
\label{sec:simple}
In this section we shall prove Theorems~\ref{thm:simple} and~\ref{thm:Rinf} (see \ssign\ref{sec:simplepf}).
In those proofs the~$2$-density of a graph will play an important r\^ole and we discuss this parameter first.

\subsection{The 2-density of a graph}
\label{sec:2density}
For every graph~$F$ with at least one edge we define its \emph{$2$-density}
\[
	d_2(F)
	=
	\begin{cases}
		\frac{|E(F)|-1}{|V(F)|-2}\,, & \text{if}\ |V(F)|\geq 3\,, \\
		1\,,                         & \text{if}\ F=K_2\,,
	\end{cases}
\]
and denote by $m_2(F)$ the \emph{maximum~$2$-density}
\[
	m_2(F)
	=
	\max\big\{d_2(F')\colon F'\subseteq F \tand |E(F')|\geq 1\big\}\,.
\]
We say a graph~$F$ with at least two edges is \emph{strictly~$2$-balanced} if $d_2(F)>d_2(F')$
for every proper subgraph $F'\subsetneq F$ with at least one edge. In particular, for any strictly~$2$-balanced graph~$F$ we have
$d_2(F)>1$ and, consequently, $|E(F)|\geq |V(F)|$.

The~$2$-density is a well-known parameter in random graph theory, as $p=n^{-1/d_2(F)}$
marks the edge probability in $G(n,p)$ where the expected number of copies of~$F$
is as large as the expected number of edges.

We first observe that building forests of copies does not increase the maximum~$2$-density.
\begin{fact}
	\label{prop:m2forest}
	For every nonempty forest of copies~$\ccF$ of a graph~$F$ with at least one edge we have
	$m_2(\fF)=m_2(F)$ for the graph $\fF=\bigcup \ccF$.
\end{fact}
\begin{proof}
	The statement is obvious if~$\ccF$ contains only one copy of~$F$ and the general case
	follows from a simple inductive argument on the number of copies in~$\ccF$.

	For some $t>1$ let $\ccF=\{F_1,\dots,F_t\}$ be a forest of copies of~$F$ and set
	\[
		q=V(F_t)\cap \bigcup_{i<t}V(F_i)\,.
	\]
	Recall that the set~$q$ contains at most two vertices and if $|q|=2$, then it spans an edge from the set $E(F_t)\cap \bigcup_{i<t}E(F_i)$.

	We shall show that $d_2(J)\leq m_2(F)$ for every subgraph $J\subseteq\bigcup\ccF$.
	Owing to the induction assumption, it suffices to consider such a graph~$J$ with both subgraphs
	$J_1=J\cap \bigcup_{i<t}F_i$ and $J_2=J\cap F_t$ containing at least one edge.
	We set $q_J=q\cap V(J)$. Note that without loss of generality we may assume that~$J$ is an induced subgraph and, hence, if $|q_J|=|q|=2$, then
	the edge spanned by~$q_J$ is also contained in~$J$. Let $\zeta\in\{0,1\}$ indicate if~$q_J$ spans an edge. We have
	\[
		1-\zeta\le (2-|q_J|)\cdot m_2(F)
	\]
	because if $\zeta=0$, then $|q_J|\le 1\le m_2(F)$.
	For $j\in\{1,2\}$ the estimate $d_2(J_j)\leq m_2(F)$
	yields
	\[
		|E(J_j)|-1 \leq \big(|V(J_j)|-2\big)\cdot m_2(F)\,.
	\]
	By adding all three inequalities we infer
	\[
		|E(J_1)|+|E(J_2)|-1-\zeta
		\le
		\bigl(|V(J_1)|+|V(J_2)|-2-|q_J|\bigr)\cdot m_2(F)\,,
	\]
	i.e., $|E(J)|-1\le \bigl(|V(J)|-2\bigr)\cdot m_2(F)$, whence $d_2(J)\le m_2(F)$.
\end{proof}

For strictly~$2$-balanced graphs similar calculations yield the following relation.

\begin{fact}\label{prop:25connected}
	For any union $H=F_1\cup F_2$ of two distinct copies~$F_1$ and~$F_2$ of a strictly~$2$-balanced
	graph~$F$ that intersect in at least two edges, we have $d_2(H)>d_2(F)$.
\end{fact}
\begin{proof}
	We denote by $I=\big(V(F_1)\cap V(F_2), E(F_1)\cap E(F_2)\big)$ the intersection of both copies and our assumption
	yields $|E(I)|\geq 2$ and $|V(I)|\geq 3$.
	Since~$F$ is strictly~$2$-balanced, we have $d_2(I)<d_2(F)$, whence
	\[
		|E(I)|-1<\bigl(|V(I)|-2\bigr)\cdot d_2(F)\,.
	\]
	By subtracting this from $2\bigl(|E(F)|-1\bigr)=2\bigl(|V(F)|-2\bigr)\cdot d_2(F)$
	we obtain
	\begin{align*}
		|E(H)|-1
		 & =
		2|E(F)|-|E(I)|-1
		>
		\bigl(2|V(F)|-|V(I)|-2\bigr)\cdot d_2(F) \\
		 & =
		\bigl(|V(H)|-2\bigr)\cdot d_2(F)\,,
	\end{align*}
	which implies $d_2(H)>d_2(F)$.
\end{proof}

We complement Fact~\ref{prop:m2forest} with the observation that certain
`cycles of copies'
have larger~$2$-density.

\begin{fact}
	\label{prop:m2scycle}
	Let~$F$ be a graph with $d_2(F)>1$. Suppose that for some $t\ge 3$
	we have a cyclic sequence $F_1, \dots, F_t$ of copies of~$F$ such that
	no edge belongs to three of these copies and for all distinct $i, j\in \ZZ/t\ZZ$
	we have
	\[
		|E(F_i)\cap E(F_j)|=
		\begin{cases}
			1 & \text{ if } j=i\pm 1 \\
			0 & \text{ else}\,.
		\end{cases}
	\]
	Then the union $\fC=\bigcup_{i\in \ZZ/t\ZZ} F_i$ satisfies $d_2(\fC)>d_2(F)$.
\end{fact}

\begin{proof}
	Notice that $|V(\fC)|\le t(|V(F)|-2)+1$
	and $|E(\fC)|=t(|E(F)|-1)$. So by subtracting $d_2(F)>1$ from
	$t(|V(F)|-2)\cdot d_2(F)=t(|E(F)|-1)$
	we obtain
	\[
		(|V(\fC)|-2)\cdot d_2(F)<|E(\fC)|-1\,,
	\]
	whence $d_2(F)<d_2(\fC)$.
\end{proof}

The following observation appeared in slightly stronger form in the work of R\"odl and Ruci\'nski~\cite{RR93}
and it establishes a link between the~$2$-density of~$F$ and the~$2$-density of its Ramsey graphs.

\begin{fact}
	\label{prop:m2Ramsey}
	For every graph~$F$ containing a cycle and every graph~$G$ satisfying the Ramsey property $G\nilra (F)_2$
	we have $m_2(G)>m_2(F)$.
\end{fact}
\begin{proof}
	Since any given graph~$F$ containing a cycle satisfies $m_2(F)>1$, there exists a strictly~$2$-balanced subgraph $\hat F\subseteq F$
	with $d_2(\hat F)=m_2(F)$ and we fix such an~$\hat F$.

	Assume for the sake of contradiction that $m_2(G)\leq m_2(F)=d_2(\hat F)$.
	This implies that all copies of~$\hat F$ in~$G$ are induced.
	We consider the auxiliary
	$|E(\hat F)|$-uniform hypergraph~$A$ with vertex set $V(A)=E(G)$ whose hyperedges
	correspond to the copies of~$\hat F$ in~$G$.
	In view of Fact~\ref{prop:25connected} and $m_2(G)\le d_2(\hat F)$, the hypergraph~$A$
	is linear. Moreover, a shortest cycle of copies in~$A$ would consist of copies
	$\hat F_1,\dots,\hat F_t$ with $t\ge 3$ in which consecutive copies share exactly one edge,
	non-consecutive copies share none, and no edge lies in three copies, which
	Fact~\ref{prop:m2scycle} excludes. Being linear and acyclic,~$A$ admits an enumeration
	of its hyperedges in which every hyperedge meets the union of its predecessors in at most one
	vertex and, since $|E(\hat F)|\ge 3$, a greedy~$2$-colouring of its vertices without
	monochromatic hyperedges. However, this colouring signifies $G\ninlra(\hat F)_2$, which is absurd.
\end{proof}

Together with Fact~\ref{prop:m2forest} this has the
following consequence, which no longer mentions the maximum~$2$-density.

\begin{fact}\label{cor:56}
	For every graph~$F$ containing a cycle every forest~$\ccF$ of copies
	of~$F$ satisfies $\bigcup\ccF\ninlra(F)_2$. \qed
\end{fact}

\subsection{Proofs of the theorems}
\label{sec:simplepf}
In this section we deduce Theorems~\ref{thm:simple} and~\ref{thm:Rinf}. Both
results are consequences of the girth Ramsey theorem
combined with
Fact~\ref{cor:56}.

\begin{proof}[Proof of Theorem~\ref{thm:simple}]
	Given a graph~$F$ containing a cycle and integers $r\geq2$ and $n\geq 1$
	let the graph~$G$ and a system of copies~$\ccG$ be provided by Theorem~\ref{thm:grt}
	applied with~$F$, $r$, and $\l=\binom{n}{2}$. Without loss of generality we may assume $G=\bigcup\ccG$
	and obviously~$G$ enjoys the Ramsey property $G\lra(F)_r$.

	Let $H\subseteq G$ be a subgraph on at most~$n$ vertices and for every edge~$e$ of~$H$ fix a copy~$F_e$ of~$F$ from the
	system~$\ccG$
	containing~$e$. Theorem~\ref{thm:grt} tells us that the set $\{F_e\colon e\in E(H)\}$ is contained in a
	forest~$\ccF$ of copies of~$F$ and by definition we have $H\subseteq \bigcup\ccF$. On the other hand,
	Fact~\ref{cor:56}
	implies $\bigcup\ccF\ninlra (F)_2$, and, since $H\subseteq \bigcup\ccF$, the same holds for~$H$.
\end{proof}

\begin{proof}[Proof of Theorem~\ref{thm:Rinf}]
	Suppose there are only finitely many minimal Ramsey graphs for a given graph~$F$
	containing a cycle and that $n\in\NN$ exceeds the number of vertices of each
	such graph. By Theorem~\ref{thm:simple} applied to~$F$, $r=2$, and~$n$ there
	is a Ramsey graph $G\lra (F)_2$ containing none of the minimal Ramsey graphs
	for~$F$, which is absurd.
\end{proof}

\section{Ramsey graphs with large free subgraphs}
\label{sec:free}
In this section we prove Theorem~\ref{thm:free}.
In the proof we employ a version of the girth Ramsey theorem for ordered graphs with an
imposed edge partition. Such a result can be deduced from an iterated application of
Theorem~\ref{thm:ogrt} following the approach of Ne\v set\v ril and R\"odl~\cites{NR77a,NR77b},
and we refer to the related work of Abramson and Harrington~\cite{AH78}. We introduce the relevant notation
and the result in \ssign\ref{sec:epgrt}.
This result
is combined with a suitable shift hypergraph, chosen in such a way
that its independent sets correspond to \(F\)-free subgraphs in the edge-partitioned graphs. We present
the details of that part of the proof in \ssign\ref{sec:shift}.
In \ssign\ref{sec:freepf} we combine the two ingredients and deduce first
Proposition~\ref{cor:free} and then Theorem~\ref{thm:free} itself.

\subsection{Girth Ramsey theorem for edge-partitioned, ordered graphs}
\label{sec:epgrt}
An \emph{edge-partitioned}, ordered graph \(F\) is an ordered
graph \((V,E,<)\) together with a fixed partition
\[
	E=E^1\dcup\dotsb\dcup E^s
\]
of its edge set for some integer \(s\ge 1\), where empty classes are allowed, and in what follows we usually suppress
the word `ordered'. For
each \(j\in[s]\) we call \(F^j=(V,E^j,<)\) the~\(j^{\text{th}}\)
graph of the edge-partitioned graph~\(F\). Given two edge-partitioned graphs \(F\) and~\(H\)
whose edge sets are partitioned into the same number~\(s\) of classes, we say that \(F\) is
an \emph{induced edge-partitioned subgraph} of \(H\) if for every \(j\in [s]\)
the ordered graph~\(F^j\) is an induced ordered subgraph of~\(H^j\). Moreover, \(F\)
and~\(H\) are \emph{isomorphic} if the order-preserving bijection from~\(V(F)\)
to~\(V(H)\) is an isomorphism from~\(F^j\) to~\(H^j\) for every \(j\in[s]\).
By a copy of~\(F\) in~\(H\) we mean an induced edge-partitioned subgraph of~\(H\)
isomorphic to~\(F\). We write \(\ttcopies HF\) for the collection of
these edge-partitioned copies. Given two edge-partitioned graphs~\(F\),~\(H\), a system of
copies \(\ccH\subseteq \ttcopies HF\), and a number of colours~\(r\),
the partition relation
\[
	\ccH\lra (F)_r
\]
means that for every \(r\)-colouring of \(E(H)\) there is a
copy \(\Fs\in \ccH\) such that for every~\({j\in [s]}\) the
graph~\(\Fs^j\) is monochromatic.

We shall use the following version of the girth Ramsey theorem for edge-partitioned,
ordered graphs from~\cite{BMRRSS}.

\begin{thm}\label{thm:epg-girth}
	For every edge-partitioned, ordered graph \(F\) and all \(r,\l\ge 2\) there exists
	an edge-partitioned graph \(H\) with the same number of classes as~\(F\) together with a
	system \(\ccH\subseteq\ttcopies{H}{F}\)
	such that \(\ccH\lra (F)_r\) and for every \(\ccN\subseteq \ccH\)
	with \(|\ccN|\in [2, \l]\) there exists \(\ccX\subseteq \ccH\) such
	that \(|\ccX|\le |\ccN|-2\) and \(\ccN\cup \ccX\) is a forest of copies.\footnote{Pedantically speaking, the underlying graphs of the copies
		in \(\ccN\cup\ccX\) form such a forest.} \qed
\end{thm}

\subsection{Shift hypergraphs}
\label{sec:shift}
As a second ingredient for the proof of Theorem~\ref{thm:free} we need hypergraphs with the
diametrically opposed properties of large chromatic number and an abundance of large
independent sets. For that we shall
employ shift hypergraphs, which can be traced back to the work of Erd\H{o}s and
Hajnal~\cite{EH66} and of Erd\H{o}s, Hajnal, and Szemer\'edi~\cite{EHS82}.
In the proof of
Proposition~\ref{cor:free} their vertices index the classes of an edge partition, so that a
large chromatic number forces a monochromatic copy of~\(F\), while every independent
set yields a subgraph containing no copy from the system of copies.
The next result essentially appeared in~\cite{RRS24}*{Theorem~3.3}, where it is attributed
to Ne\v set\v ril, R\"odl, and Sales~\cite{NRS24}*{\ssign5}.
We include the short proof, since we shall use the family of independent sets itself rather
than the fractional property that~\cite{RRS24} derives from it.

\begin{prop}\label{thm:shift}
	For all integers \(k, r\ge 2\) and every real \(\mu\in\bigl(0, \frac{k-1}k\bigr)\) there exist
	a \(k\)-uniform hypergraph \(S\) with \(\chi(S)>r\), an index set \(\Pi\),
	and a family \((I_\pi)_{\pi\in\Pi}\) of independent subsets of \(S\) such that
	\[
		\big|\bigl\{\pi\in\Pi\colon x\in I_\pi\bigr\}\big|\ge \mu|\Pi|
	\]
	holds for every \(x\in V(S)\).
\end{prop}

\begin{proof}
	Since \(\mu<\frac{k-1}k\), there is an integer \(q\ge 1\)
	with \(q(k-1)\ge \mu(qk+1)\). We set \(m=qk+1\),
	and let \(n\) be sufficiently large such that \(n\lra (k+m-1)^m_r\). We define
	the \(k\)-uniform shift hypergraph \(S=\mathrm{Sh}^{(k)}(n, m)\) with vertex
	set \([n]^{(m)}\) to have the following \(\binom n{k+m-1}\) edges:
	for every increasing sequence \(x_1<\dots<x_{k+m-1}\) from \([n]\)
	we include the edge
	\[
		\bigl\{\{x_1, \dots, x_m\}, \{x_2, \dots, x_{m+1}\}, \dots,
		\{x_k, \dots, x_{k+m-1}\}\bigr\}
	\]
	in \(E(S)\). Evidently, our choice of \(n\) yields \(\chi(S)>r\). Moreover,
	our choice of~\(m\) guarantees that the set
	\[
		J=\big\{j\in [m]\colon j\not\equiv 1\pmod{k}\big\}
	\]
	satisfies
	\[
		|J|
		=m-(q+1)
		=q(k-1)
		\ge \mu m\,.
	\]

	For every vertex \(a=\{x_1, \dots, x_m\}\) of \(S\) with \(x_1<\dots<x_m\)
	and every permutation \(\pi\) of~\([n]\) we denote the unique index \(j\in [m]\)
	satisfying
	\[
		\pi(x_j)=\max\big\{\pi(x_i)\colon i\in [m]\big\}
	\]
	by~\(\nu(a, \pi)\).

	Now for every permutation \(\pi\) we shall observe that the set
	$
		I_\pi=\{a\in V(S)\colon \nu(a, \pi)\in J\}
	$
	is independent in \(S\). Indeed, consider an edge given by \(x_1<\dots<x_{k+m-1}\)
	and set \(a_i=\{x_i, \dots, x_{i+m-1}\}\) for \(i\in[k]\). If \(i\in[k-1]\) satisfies \(\nu(a_i, \pi)\ne 1\),
	then the maximum of~\(\pi\) on~\(a_i\) is attained at an
	element of \(a_i\setminus\{x_i\}\subseteq a_{i+1}\), whose index in~\(a_{i+1}\)
	is \(\nu(a_i, \pi)-1\), while the only element of~\(a_{i+1}\) not in~\(a_i\) is
	its largest element~\(x_{i+m}\). Consequently, we have
	\(\nu(a_{i+1}, \pi)\in\{\nu(a_i, \pi)-1, m\}\) and, as \(1, m\notin J\), the
	assumption \(a_1, \dots, a_k\in I_\pi\) would force \(\nu(a_1, \pi), \dots, \nu(a_k, \pi)\)
	to be \(k\) consecutive integers in~\(J\). This is absurd, since one of those~\(k\)
	consecutive integers is congruent to~\(1\) modulo~\(k\).

	For the index set \(\Pi\) consisting of all permutations of \([n]\) the
	family \((I_\pi)_{\pi\in\Pi}\) has the desired property, since
	for every vertex \(a\in V(S)\) we have
	\[
		\big|\bigl\{\pi\in\Pi\colon a\in I_\pi\bigr\}\big|
		=
		\sum_{j\in J} \big|\bigl\{\pi\in\Pi\colon \nu(a, \pi)=j\bigr\}\big|
		=
		|J|\cdot\frac{|\Pi|}{m}
		\ge
		\mu |\Pi|\,. \qedhere
	\]
\end{proof}

\subsection{Proof of the theorem}
\label{sec:freepf}
Combining Theorem~\ref{thm:epg-girth} with Proposition~\ref{thm:shift} we first obtain
the following statement, in which the system of copies is still present.

\begin{prop}\label{cor:free}
	For every graph \(F\) with \(|E(F)|=k\ge 2\), all integers \(r\), \(\l\ge 2\),
	and every real \(\mu\in\bigl(0, \frac{k-1}k\bigr)\) there exist a graph~\(H\)
	and a system of copies \(\ccH\subseteq \binom HF\) with $H=\bigcup\ccH$ such that
	\begin{enumerate}[label=\alabel]
		\item\label{it:free-ramsey} \(\ccH\lra (F)_r\),
		\item\label{it:free-forest} for every \(\ccN\subseteq\ccH\) with \(|\ccN|\leq \l\)
		      there is some \(\ccX\subseteq\ccH\) such that \(\ccN\cup\ccX\)
		      is a forest of copies,
		\item\label{it:free-dense} every subgraph \(G\) of \(H\) has a subgraph \(G'\) with
		      \(|E(G')|\ge \mu|E(G)|\) such that no copy in~\(\ccH\) is a subgraph
		      of~\(G'\),
		\item\label{it:free-girth} and every induced cycle of~$H$ with at most~$\l$ edges
		      is contained in a unique member of~$\ccH$, and if $\girth(F)<\infty$, then
		      $\girth(H)=\girth(F)$.
	\end{enumerate}
	Moreover, if~$F$ is an ordered graph, then~$H$ can be ordered in such a way
	that $\ccH\subseteq\binom HF_{\!<}$.
\end{prop}

\begin{proof}
	Since properties~\ref{it:free-forest} and~\ref{it:free-girth} only become stronger
	as~$\l$ increases, we may assume that $\l\ge\girth(F)$ if $\girth(F)<\infty$.
	Let \(S\) be the hypergraph provided by Proposition~\ref{thm:shift} applied with \(k\), \(r\), and~\(\mu\), and let
	\((I_\pi)_{\pi\in\Pi}\) be the family of independent sets.
	Moreover, we may assume \(V(S)=[s]\).

	Let~\(F_\ep\) be an edge-partitioned graph whose underlying graph is the vertex-disjoint
	union of copies~\(F_e\) of~\(F\), one for every \(e\in E(S)\), such that the~\(k\) edges of
	each~\(F_e\) belong to the~\(k\) distinct graphs~\(F_\ep^j\) with \(j\in e\). In other words,
	the edge set of~\(F_\ep\) is partitioned into \(s\) classes \(E^1\dcup\dotsb\dcup E^s\)
	indexed by the vertices of~\(S\),
	and the \(k\)-element edge set of every individual copy~\(F_e\) with \(e\in E(S)\)
	contains one edge from~\(E^j\) for every vertex~\(j\) of~\(e\). The vertex set
	of~$F_\ep$ is ordered arbitrarily, except that, if~$F$ is ordered, we require
	every~$F_e$ to be an ordered copy of~$F$.

	Next, Theorem~\ref{thm:epg-girth} applied to~\(F_\ep\), \(r\), and~\(\l\) yields an
	edge-partitioned graph~\(H_\ep\) and a system of
	copies \(\ccH_\ep\subseteq \ttcopies{H_\ep}{F_\ep}\) with \(\ccH_\ep\lra(F_\ep)_r\) such
	that for every \(\ccN\subseteq\ccH_\ep\) with \(|\ccN|\in[2,\l]\) there is
	some \(\ccX\subseteq\ccH_\ep\) for which \(\ccN\cup\ccX\) is a forest of copies.

	Now let~\(\ccH\) be obtained
	from~\(\ccH_\ep\) by resolving each copy \(\Fs\in\ccH_\ep\) into the \(|E(S)|\) copies
	of~\(F\) corresponding to the graphs~\(F_e\), i.e., into the images of the graphs~\(F_e\)
	under the order-preserving bijection from~\(V(F_\ep)\) to~\(V(\Fs)\), and let
	$H=\bigcup\ccH$, which is a subgraph of the underlying graph of~$H_\ep$. Every edge
	of~\(H\) inside~\(V(\Fs)\) belongs to some~\(H_\ep^j\) and, hence, to~\(\Fs^j\), and
	each~\(F_e\) is induced in the vertex-disjoint union~\(F_\ep\). Consequently, the members
	of~\(\ccH\) are induced in~\(H\), i.e., \(\ccH\subseteq\binom HF\), and if~$F$ is
	ordered, then the ordering inherited from~$H_\ep$ yields $\ccH\subseteq\binom HF_{\!<}$.
	Below we verify that \(H\) and \(\ccH\) enjoy the
	properties~\ref{it:free-ramsey}--\ref{it:free-girth}.

	For property~\ref{it:free-ramsey} we consider an \(r\)-colouring of~\(E(H)\), which we
	extend arbitrarily to the remaining edges of~$H_\ep$. Owing
	to \(\ccH_\ep\lra(F_\ep)_r\) there is a copy \(\Fs\in\ccH_\ep\) such that for
	every \(j\in[s]\) the graph~\(\Fs^j\) is monochromatic, say in colour~\(c(j)\), where
	we choose~\(c(j)\) arbitrarily if~\(\Fs^j\) has no edges.
	As \(\chi(S)>r\), some edge \(e\in E(S)\) is monochromatic with respect to~\(c\), and
	the copy of~\(F\) in~\(\Fs\) corresponding to~\(F_e\) is a monochromatic member
	of~\(\ccH\).

	For property~\ref{it:free-forest} we first observe that resolving the members of a
	forest of copies of~\(F_\ep\) yields a forest of copies of~\(F\). Indeed, since the
	graphs~\(F_e\) are vertex-disjoint and every edge of~\(F_\ep\) belongs to one of them,
	the copies of~\(F\) into which a member is resolved are vertex-disjoint and their edge sets
	partition the edge set of this member. So if we list, following an admissible
	enumeration of the given forest, the copies of~\(F\) obtained from each member one
	after another, then every copy meets the union of its predecessors either not at all
	or in the connector~\(q\) of the member it is obtained from, and in the latter
	case~\(q\) is a vertex or an edge of this copy and of the union of its predecessors.
	Since copies of~\(F\) have more than two vertices, no copy is listed twice and this
	enumeration is admissible. Now given \(\ccN\subseteq\ccH\) with \(|\ccN|\leq \l\), we
	fix for every member of~\(\ccN\) a copy in~\(\ccH_\ep\) from which it arises. These at
	most~\(\l\) copies belong to a forest of copies \(\ccF_\ep\subseteq\ccH_\ep\) by the forest
	property of~\(\ccH_\ep\), or trivially if there is at most one such copy. Resolving the
	members of~\(\ccF_\ep\) leads to a forest of copies \(\ccF\subseteq\ccH\)
	containing~\(\ccN\), and \(\ccX=\ccF\setminus\ccN\) is as required.

	For property~\ref{it:free-dense} we consider an arbitrary subgraph~\(G\) of \(H\).
	For each index \(\pi\in\Pi\), the spanning subgraph \(G_\pi\) of \(G\) with
	\[
		E(G_\pi)=\bigcup_{j\in I_\pi} E(G)\cap E(H_\ep^j)
	\]
	contains no member of~\(\ccH\). Indeed, every member of~\(\ccH\) corresponds to
	some~\(F_e\) with \(e\in E(S)\) and, hence, has an edge in~\(H_\ep^j\) for
	every \(j\in e\), whereas the independent set~\(I_\pi\) misses some \(j\in e\).
	Moreover, since
	\[
		\sum_{\pi\in\Pi} |E(G_\pi)|
		=
		\sum_{j\in V(S)} \big|E(G)\cap E(H_\ep^j)\big| \cdot
		\big|\bigl\{\pi\in \Pi\colon j\in I_\pi\bigr\}\big|
		\ge
		\mu |E(G)|\cdot|\Pi|\,,
	\]
	at least one of these graphs can play the r\^{o}le of \(G'\).

	It remains to verify property~\ref{it:free-girth}. Owing to $H=\bigcup\ccH$, the edges
	of an induced cycle~$C$ of~$H$ with at most~$\l$ edges are covered by at most~$\l$
	copies from~$\ccH$, which belong to a forest of copies $\ccF\subseteq\ccH$ by
	property~\ref{it:free-forest}. As~$C$ is an induced cycle of~$\bigcup\ccF$,
	Corollary~\ref{cor:indcycle} shows that~$C$ is contained in a member of~$\ccF$.
	If~$C$ were contained in two distinct members of~$\ccH$, then by
	property~\ref{it:free-forest} they would belong to a forest of copies, and in an
	admissible enumeration of this forest the later of the two would share at most two
	vertices with the earlier one, which is absurd, since~$C$ has at least three
	vertices. Finally, if $\girth(F)<\infty$, then $\girth(F)\le\l$ and, as~$\ccH$ is
	nonempty by property~\ref{it:free-ramsey}, the graph~$H$ contains a cycle with $\girth(F)$
	edges, and so a shortest cycle of~$H$ has at most~$\l$ edges. Being induced, it is
	contained in a member of~$\ccH$, which shows $\girth(H)=\girth(F)$.
\end{proof}

It remains to observe that for the graphs~\(F\) in Theorem~\ref{thm:free} the system
\(\ccH\) in Proposition~\ref{cor:free} can be dispensed with.

\begin{proof}[Proof of Theorem~\ref{thm:free}]
	For a given \((e,v)\)-inseparable graph~\(F\)
	set \(k=|E(F)|\) and let \(r\ge 2\) and \(\eps>0\) be fixed, where we may
	assume that \(\mu=1-\frac1k-\eps\) is positive. Proposition~\ref{cor:free} applied
	with~\(r\), \(\l=k\), and~\(\mu\) yields a graph, which we denote by~\(G\), and a system
	of copies \(\ccH\subseteq\binom GF\) with $G=\bigcup\ccH$ and the
	properties~\ref{it:free-ramsey}--\ref{it:free-girth}.

	We contend that every subgraph \(F_\star\subseteq G\) isomorphic to~\(F\) is a member
	of~\(\ccH\). Indeed, the \(k\) edges of~\(F_\star\) are covered by at most~\(k\) copies
	from~\(\ccH\), which belong to a forest of copies \(\ccF\subseteq\ccH\) by
	property~\ref{it:free-forest}.
	Since \(F_\star\subseteq\bigcup\ccF\), Proposition~\ref{cor:nonew}, applied to the
	graph~\(F\), shows \(F_\star\in\ccF\subseteq\ccH\).

	Now property~\ref{it:free-ramsey} yields \(G\lra(F)_r\) and, since~\(F\) contains a
	cycle owing to \(\delta(F)\ge 2\) (see Fact~\ref{fact:evins}),
	property~\ref{it:free-girth} yields \(\girth(G)=\girth(F)\). Moreover, for every
	subgraph \(H\subseteq G\) property~\ref{it:free-dense}, applied to~\(H\), provides a
	subgraph \(H'\subseteq H\) containing no copy from~\(\ccH\). By the previous
	paragraph~\(H'\) is, therefore, \(F\)-free, while
	\[
		|E(H')|
		\ge\mu|E(H)|
		=\Big(1-\frac1{|E(F)|}-\eps\Big)|E(H)|\,,
	\]
	which concludes the proof of Theorem~\ref{thm:free}.
\end{proof}

\section{Generalised Sidon sets}
\label{sec:sidon}
For the proof of Theorem~\ref{thm:sidon} we follow the approach of~\cite{RRS26}, which
is based on generalised theta graphs. Given integers $k, h\ge 2$, the
\emph{generalised theta graph}~$\Theta_{k,h}$ consists of~$h$ internally
vertex-disjoint paths of length~$k$ with common end vertices. Thus~$\Theta_{k,h}$ has
$(k-1)h+2$ vertices, say~$a_0$, $a_k$, and~$a^{(j)}_i$ with
$(i,j)\in[k-1]\times[h]$, and~$kh$ edges such that $a_0a^{(j)}_1\dots a^{(j)}_{k-1}a_k$
is a path for every $j\in[h]$ (see Figure~\ref{fig:theta}).

\begin{figure}[ht]
	\centering
	\begin{tikzpicture}[scale=1.1]
		\coordinate (x0) at (0.5,0);
		\coordinate (x1) at (2.5,1.6);
		\coordinate (x2) at (5,2.4);
		\coordinate (x3) at (7.5,2.4);
		\coordinate (x4) at (10,1.6);
		\coordinate (x5) at (2.5,0);
		\coordinate (x6) at (5,0);
		\coordinate (x7) at (7.5,0);
		\coordinate (x8) at (10,0);
		\coordinate (x9) at (2.5,-1.6);
		\coordinate (x10) at (5,-2.4);
		\coordinate (x11) at (7.5,-2.4);
		\coordinate (x12) at (10,-1.6);
		\coordinate (x13) at (12,0);
		\draw[line width=1.25pt]
		(x0) -- node[above,sloped] {$<$}
		(x1) -- node[above,sloped] {$<$}
		(x2) -- node[above,sloped] {$<$}
		(x3) -- node[above,sloped] {$<$}
		(x4) -- node[above,sloped] {$<$}
		(x13);
		\draw[line width=1.25pt]
		(x0) -- node[above,sloped] {$<$}
		(x5) -- node[above,sloped] {$<$}
		(x6) -- node[above,sloped] {$<$}
		(x7) -- node[above,sloped] {$<$}
		(x8) -- node[above,sloped] {$<$}
		(x13);
		\draw[line width=1.25pt]
		(x0) -- node[below,sloped] {$<$}
		(x9) -- node[below,sloped] {$<$}
		(x10) -- node[below,sloped] {$<$}
		(x11) -- node[below,sloped] {$<$}
		(x12) -- node[below,sloped] {$<$}
		(x13);
		\foreach \i in {0,...,13}{\hvertex{(x\i)}}
		\node[anchor=east] at (x0) {$a_0$};
		\node[anchor=west] at (x13) {$a_5$};
		\node[anchor=south] at (x1) {$a_1^{(1)}$};
		\node[anchor=south] at (x2) {$a_2^{(1)}$};
		\node[anchor=south] at (x3) {$a_3^{(1)}$};
		\node[anchor=south] at (x4) {$a_4^{(1)}$};
		\node[anchor=south] at (x5) {$a_1^{(2)}$};
		\node[anchor=south] at (x6) {$a_2^{(2)}$};
		\node[anchor=south] at (x7) {$a_3^{(2)}$};
		\node[anchor=south] at (x8) {$a_4^{(2)}$};
		\node[anchor=north] at (x9) {$a_1^{(3)}$};
		\node[anchor=north] at (x10) {$a_2^{(3)}$};
		\node[anchor=north] at (x11) {$a_3^{(3)}$};
		\node[anchor=north] at (x12) {$a_4^{(3)}$};
	\end{tikzpicture}
	\caption{An ascending theta graph~$\Theta_{5,3}$.}
	\label{fig:theta}
\end{figure}

For a graph~$G$ with
$V(G)\subseteq\NN$ we call a copy of~$\Theta_{k,h}$ in~$G$ \emph{ascending} if its
vertices can be labelled as above in such a way that $a_0<a^{(j)}_1<\dots<a^{(j)}_{k-1}<a_k$
holds for every $j\in[h]$. Setting $a^{(j)}_0=a_0$ and $a^{(j)}_k=a_k$, the~$h$
sums $\sum_{i=1}^k\big(a^{(j)}_i-a^{(j)}_{i-1}\big)$ with $j\in[h]$ telescope
to~$a_k-a_0$.

Thus, for the set $X=\{b-a\colon ab\in E(G)\tand a<b\}$, every ascending
copy of~$\Theta_{k,h}$ in~$G$ gives rise to~$h$ representations of~$a_k-a_0$ as a
sum of~$k$ terms from~$X$, which are distinct if the vertices of~$G$ are
sufficiently far apart.

The set~$X$ promised in Theorem~\ref{thm:sidon} is obtained this way
from disjoint copies of ordered Ramsey graphs~$H_r$ provided by Proposition~\ref{cor:free}
applied to~$\Theta_{k,h}$ for every number of colours $r\geq 2$.

\begin{prop}\label{prop:sidon-graph}
	For all integers $k, h\ge 2$ and $\l\ge 2k$ and every real
	$\mu\in\bigl(0, 1-\frac1{kh}\bigr)$ there is an infinite graph~$G$ with
	$V(G)\subseteq\NN$ satisfying the following properties.
	\begin{enumerate}[label=\alabel]
		\item\label{it:sg-ramsey} For every colouring of~$E(G)$ with finitely many
		      colours there is a monochromatic, induced, ascending copy of~$\Theta_{k,h}$
		      in~$G$.
		\item\label{it:sg-cycle} Every induced cycle of~$G$ with at most~$\l$ edges
		      has exactly~$2k$ edges and is contained in a unique copy of~$\Theta_{k,h}$
		      in~$G$, which is ascending.
		\item\label{it:sg-free} The graph~$G$ contains no copy of~$\Theta_{k,h+1}$.
		\item\label{it:sg-dense} Every finite subgraph $G_0\subseteq G$ has a
		      subgraph~$G'$ with $|E(G')|\ge\mu|E(G_0)|$ containing no copy
		      of~$\Theta_{k,h}$.
	\end{enumerate}
\end{prop}

Properties~\ref{it:sg-ramsey} and~\ref{it:sg-cycle} correspond
to~\cite{RRS26}*{Proposition~2.1}, and property~\ref{it:sg-free} is established in its
proof. The new ingredient is property~\ref{it:sg-dense}, which is a consequence of
Proposition~\ref{cor:free}.

\begin{proof}
	We fix an ordering of~$\Theta_{k,h}$ in which
	$a_0<a^{(j)}_1<\dots<a^{(j)}_{k-1}<a_k$ holds for every $j\in[h]$, and we set
	$\l'=\max\{\l, k(h+1)\}$. For every integer $r\ge 2$,
	Proposition~\ref{cor:free} applied to the ordered graph $F=\Theta_{k,h}$, which
	has~$kh$ edges, and to~$r$, $\l'$, and~$\mu$ yields an ordered graph~$H_r$
	and a system of copies $\ccH_r\subseteq\binom{H_r}{\Theta_{k,h}}_{\!<}$ with
	$H_r=\bigcup\ccH_r$ and the properties~\ref{it:free-ramsey}--\ref{it:free-girth}
	of Proposition~\ref{cor:free}. Let~$G$ be the vertex-disjoint union of the
	graphs~$H_r$ with $r\ge 2$, where we place the vertex sets of these graphs into~$\NN$
	so that their orderings are respected.

	The systems~$\ccH_r$ consist of ordered, induced copies of~$\Theta_{k,h}$, and
	properties~\ref{it:free-ramsey} and~\ref{it:free-forest} of Proposition~\ref{cor:free}
	provide everything that the proof of~\cite{RRS26}*{Proposition~2.1} uses from the
	ordered girth Ramsey theorem~\cite{RRS26}*{Theorem~2.3}. Hence, that proof carries
	over verbatim, with~$h$, $\l$, and~$\l'$ here playing the r\^oles of the
	parameters~$\l$, $s$, and~$n$ in~\cite{RRS26}. It yields properties~\ref{it:sg-ramsey}
	and~\ref{it:sg-free}, as the argument for~\ref{it:sg-free} does not use the orderings,
	and property~\ref{it:sg-cycle} with uniqueness among the ascending copies
	of~$\Theta_{k,h}$. We omit the details.

	We show that every copy~$\Theta$ of~$\Theta_{k,h}$
	in~$G$ is a member of one of the systems~$\ccH_r$. Being connected,~$\Theta$ is contained
	in one of the graphs~$H_r$, which has girth~$2k$ by part~\ref{it:free-girth} of
	Proposition~\ref{cor:free}. Hence, for every $j\in[2,h]$ the cycle formed by the first
	and the~$j^{\text{th}}$ path of~$\Theta$ is induced in~$H_r$, and owing to $2k\le\l'$
	property~\ref{it:free-girth} also shows that it is contained in a member of~$\ccH_r$. Any two of
	these members share the $k+1\ge 3$ vertices of the first path of~$\Theta$ and, in view of
	property~\ref{it:free-forest}, they must
	coincide. Consequently,~$\Theta$ is contained in
	one of these members and, comparing the numbers of edges, we infer~$\Theta\in\ccH_r$. Together with
	property~\ref{it:free-girth} of Proposition~\ref{cor:free} this shows the uniqueness in~\ref{it:sg-cycle}.

	Finally, a finite subgraph $G_0\subseteq G$ has edges in only finitely many of the
	graphs~$H_r$. Applying property~\ref{it:free-dense} of Proposition~\ref{cor:free} to
	each of the graphs~$G_0\cap H_r$, we obtain subgraphs whose union~$G'$ satisfies
	$|E(G')|\ge\mu|E(G_0)|$ and contains no member of any of the systems~$\ccH_r$.
	Since~$G'$ contains no copy of~$\Theta_{k,h}$, this establishes property~\ref{it:sg-dense}.
\end{proof}

\begin{proof}[Proof of Theorem~\ref{thm:sidon}]
	As the argument follows the proof of~\cite{RRS26}*{Theorem~1.1} very closely, we
	only sketch it. Let~$G$ be the graph provided by Proposition~\ref{prop:sidon-graph}
	applied with $\l=2k$. A shortest cycle of~$G$ is induced, and so
	property~\ref{it:sg-cycle} tells us that~$G$ has no cycle with fewer than~$2k$
	edges. Consequently, every cycle of~$G$ with~$2k$ edges is induced and, hence,
	contained in a unique copy of~$\Theta_{k,h}$ in~$G$, which is ascending. Together
	with properties~\ref{it:sg-ramsey} and~\ref{it:sg-free}, these are the properties of the
	auxiliary graph used in~\cite{RRS26}*{\ssign3}, where, after an order-preserving
	relabelling of the vertices by the powers of $m=2k+1$, the set
	$X=\{b-a\colon ab\in E(G)\tand a<b\}$ is considered.

	The map $ab\longmapsto|b-a|$ is
	a bijection between~$E(G)$ and~$X$, and~\cite{RRS26}*{Claim~3.2} shows that any
	two distinct representations of an integer~$n$ as a sum of~$k$ terms from~$X$
	correspond to two of the~$h$ paths of an ascending copy of~$\Theta_{k,h}$
	in~$G$, whose end vertices~$a_0$ and~$a_k$ are determined by $n=a_k-a_0$.
	Consequently, if~$n$ has $t\ge 2$ representations as a sum of~$k$ terms from some
	set $Y'\subseteq X$, then the corresponding edges form~$t$ ascending paths of
	length~$k$ between the same two vertices. As~$G$ has no cycle with fewer
	than~$2k$ edges, these paths are internally vertex-disjoint and, hence, they form a
	copy of~$\Theta_{k,t}$ in~$G$.

	Now property~\ref{it:sg-free} shows $\rho_k(X)\le h$, while
	property~\ref{it:sg-ramsey} yields $X\lra[k,h]_r$ for every~$r$ and, in
	particular, $\rho_k(X)\ge h$. So~$X$ is an infinite~$B_{k,h}$-set
	satisfying~\ref{it:sidon-ramsey}. For~\ref{it:sidon-dense} we consider a finite
	set $Y\subseteq X$ and the finite graph $G_0\subseteq G$ formed by the edges
	corresponding to the elements of~$Y$. Property~\ref{it:sg-dense} yields a
	subgraph $G'\subseteq G_0$ with $|E(G')|\ge\mu|Y|$ containing no copy
	of~$\Theta_{k,h}$, and by the previous paragraph the set $Y'\subseteq Y$
	corresponding to~$E(G')$ satisfies $\rho_k(Y')<h$.
\end{proof}

\subsection*{Acknowledgement} We would like to thank Jonathan Tidor for interesting
conversations on Theorems~\ref{thm:forcing} and~\ref{thm:noforest}, which led to a
simplification in \ssign\ref{sec:noforest}.

\end{document}